\documentclass[11pt,a4paper]{amsart}
\usepackage{amsmath,amssymb,times, amscd, graphicx, xspace,mathrsfs,a4wide,hyperref,url,chngcntr,mathtools,tikz-cd,comment,cleveref,orcidlink}
\usepackage[title]{appendix}

\hypersetup{
  colorlinks   = true, 
  urlcolor     = blue, 
  linkcolor    = blue, 
  citecolor   = blue 
}

\newcommand{\opT}{\operatorname{T}}

\newcommand{\ord}{\operatorname{ord}}

\newcommand{\C}{\mathbb{C}}
\newcommand{\Z}{\mathbb{Z}}
\newcommand{\Q}{\mathbb{Q}}
\newcommand{\R}{\mathbb{R}}
\newcommand{\N}{\mathbb{N}}
\newcommand{\A}{\mathbb{A}}
\newcommand{\T}{\mathbb{T}}
\newcommand{\p}{\mathfrak{p}}

\newcommand{\tr}{\operatorname{tr}}

\newcommand{\eps}{\varepsilon}

\newcommand{\Sym}{\operatorname{Sym}}

\newcommand{\GL}{\operatorname{GL}}
\newcommand{\PGL}{\operatorname{PGL}}

\renewcommand{\O}{\mathcal{O}}

\renewcommand{\mod}{\operatorname{mod}}
\newcommand{\Mod}{\operatorname{Mod}}
\newcommand{\RMod}{\operatorname{RMod}}

\renewcommand{\phi}{\varphi}

\newcommand{\Ind}{\operatorname{Ind}}

\newcommand{\End}{\operatorname{End}}

\newcommand{\rank}{\operatorname{rank}}

\newcommand{\Hom}{\operatorname{Hom}}
\newcommand{\Ext}{\operatorname{Ext}}
\newcommand{\Tor}{\operatorname{Tor}}

\newcommand{\Ann}{\operatorname{Ann}}

\newcommand{\Spec}{\operatorname{Spec}}

\newcommand{\ad}{\operatorname{ad}}

\newcommand{\surj}{\twoheadrightarrow}

\newcommand{\colim}{\operatorname{colim}}

\newcommand{\G}{\mathbf{G}} 
\newcommand{\B}{\mathfrak{B}}
\newcommand{\m}{\mathfrak{m}}
\newcommand{\rhobar}{\bar{\rho}}

\newcommand{\CGO}{\mathfrak{C}_G(\O)}
\newcommand{\CGvO}{\mathfrak{C}_{G_v}(\O)}
\newcommand{\rhobarv}{\rhobar_v}
\newcommand{\Rpsv}{R_{\tr \rhobarv}}
\newcommand{\Rloc}{R_{\operatorname{loc}}}

\newcommand{\Rpsp}{R^{\operatorname{ps}}_p}
\newcommand{\RSrho}{R_{\mathcal{S},\rho}}
\newcommand{\Rpsphat}{(\Rpsp)_\p^\wedge}
\newcommand{\Ap}{A_\p}
\newcommand{\Aphat}{\widehat{A}_\p}
\newcommand{\Mphat}{\widehat{M}_\p}
\newcommand{\Mp}{M_\p}
\newcommand{\Htilde}{\widetilde{H}}
\newcommand{\Htildestar}{\widetilde{H}_\ast(X_{U^p},\O)}
\newcommand{\Htildei}{\widetilde{H}_i(X_{U^p},\O)}
\newcommand{\Htildeq}{\widetilde{H}_{q_0}(X_{U^p},\O)}

\newcommand{\ModfgaGzetaO}{\operatorname{Mod}_{G,\zeta}^{\operatorname{fga}}(\O)}
\newcommand{\hatotimes}{\widehat{\bigotimes}}

\newcommand{\GFS}{\Gamma_{F,S}}
\newcommand{\GFv}{\Gamma_{F_v}}
\newcommand{\GQp}{\Gamma_{\Q_p}}

	\newcommand{\IFv}{I_{F_v}}

\newcommand{\Rgl}{R_{\operatorname{gl}}}

\newcommand{\otimesL}{\otimes^{\mathbf{L}}}

\newcommand{\otimesLOK}{\otimesL_{\O[[K]]}}
\newcommand{\otimeshat}{\widehat{\otimes}}
\newcommand{\q}{\mathfrak{q}}
\newcommand{\CBS}{C^{\operatorname{BS}}_\bullet}
\newcommand{\Cad}{C^{\operatorname{ad}}_\bullet}

\newcommand{\sigmao}{{\sigma^\circ}}
\newcommand{\dpt}{\operatorname{dp}}

\renewcommand{\H}{H^1\big(\GFS,\ad^0\rho\big)} 
\newcommand{\Hp}{H^1\big(\GFv,\ad^0\rho\big)} 
\newcommand{\Hf}{H^1_f\big(\GFS,\ad^0\rho\big)} 
\newcommand{\Hfd}{H^1_f\big(\GFS,\ad^0\rho(1)\big)} 
\newcommand{\Hfp}{H^1_f\big(\GFv,\ad^0\rho\big)} 

\newcommand{\OK}{\O[[K]]}

\newcommand{\Ahat}{\widehat{A}}

\newcommand{\Dbarchi}{D_{\rhobar}^{\chi}}

\newcommand{\Dbarchivbox}{D^{\chi,\square}_{\rhobar_v}}
\newcommand{\Dbarchiv}{D_{\rhobar_v}^{\chi}}

\newcommand{\Ctildeinfty}{\widetilde{\mathcal{C}}(\infty)}
\newcommand{\HCtildeinfty}{H_{q_0}(\widetilde{\mathcal{C}}(\infty))}
\newcommand{\RModcptEB}{\operatorname{RMod}^{\operatorname{cpt}}(E_\B)}
\newcommand{\RModcptOK}{\operatorname{RMod}^{\operatorname{cpt}}(\O[[K]])}
\newcommand{\RModcptOKv}{\operatorname{RMod}^{\operatorname{cpt}}(\O[[K_v]])}

\newcommand{\fraka}{\mathfrak{a}}

\newcommand{\Oinfty}{\mathcal{O}_\infty}
\newcommand{\RSlocS}{R^{S,\operatorname{loc}}_{\mathcal{S}}}

\newcommand{\Rrhobarv}{R_{\rhobar_v}}
\newcommand{\Rrhobarvbox}{\Rrhobarv^{\square}}

\newcommand{\EB}{E_\B}
\newcommand{\ZB}{Z_\B}
\newcommand{\EBphat}{(\EB)_\p^\wedge}
\newcommand{\ZBphat}{(\ZB)_\p^\wedge}

\newcommand{\PB}{P_\B}

\newcommand{\mzerophat}{(M_0)_\p^\wedge}
\newcommand{\mzero}{M_0}

\newcommand{\Sinfphat}{(S_\infty)_{\q_\infty}^\wedge}
\newcommand{\Rinfphat}{(R_\infty)_{\p_\infty}^\wedge}
\newcommand{\Rrhop}{R_{p,\rho}}
\newcommand{\Rpsigma}{\Rrhop(\sigma)}

\newcommand{\wtchi}{(\lambda,\tau,\chi)}
\newcommand{\wtchiv}{{(\lambda_v,\tau_v,\chi_v)}}

\newcommand{\RSphat}{{(\RS)}_\p^\wedge}

\newcommand{\stv}{\sigma(\tau_v)}
\newcommand{\swtv}{\sigma(\lambda_v,\tau_v)}
\newcommand{\swtp}{\sigma(\lambda,\tau)}
\newcommand{\sowtp}{\sigma^\circ(\lambda,\tau)}

\newcommand{\RS}{R_{\mathcal{S}}}

\newcommand{\Dchi}{D_{\rho}^{\chi}}

\newcommand{\Dchivbox}{D^{\chi,\square}_{\rho_v}}
\newcommand{\Dchiv}{D_{\rho_v}^{\chi}}

\newcommand{\RpsB}{R^{\zeta\eps}_{\operatorname{tr}\rhobar_\B}}

\newcommand{\CBO}{\mathfrak{C}_\B(\O)}

\newcommand{\ModpfaGzetaO}{\Mod^{\operatorname{pfa}}_{G,\zeta}(\O)}
\newcommand{\ModpfaGO}{\Mod^{\operatorname{pfa}}_{G}(\O)}
\newcommand{\Rrho}{R_\rho}

\newcommand{\Rrhov}{R_{\rho_v}}

\newcommand{\DS}{D_{\mathcal{S}}}

\newcommand{\Dv}{D_v}

\newcommand{\astreck}{\underline{a}}

\newcommand{\Kbulleta}{K_{\bullet}(\astreck)}

\newcommand{\ArtE}{\operatorname{Art}_E}
\newcommand{\CNLE}{\operatorname{CNL}_E}
\newcommand{\Artk}{\operatorname{Art}_k}
\newcommand{\CNLk}{\operatorname{CNL}_k}

\newcommand{\Dtrrhobarvchi}{D^{\chi_v}_{\tr \rhobar_v}}
\newcommand{\Rrhovbox}{\Rrhov^{\square}}

\newcommand{\Rbarv}{\overline{R}_{\rhobarv}}
\newcommand{\Rtrrhobarv}{R_{\operatorname{tr}\rhobarv}}
\newcommand{\ModladmGzetaO}{\Mod^{\operatorname{ladm}}_{G,\zeta}(\O)}
\newcommand{\CGzetaO}{\mathfrak{C}_{G,\zeta}(\O)}

\newcommand{\Ctilde}{\widetilde{C}_\bullet}
\newcommand{\minfphat}{(M_\infty)_{\p_\infty}^\wedge}

\newcommand{\asigmap}{\mathfrak{a}_{\sigma,\p}}

\DeclareFontFamily{U}{mathc}{}
\DeclareFontShape{U}{mathc}{m}{it}%
{<->s*[1.03] mathc10}{}

\DeclareMathAlphabet{\mathscr}{U}{mathc}{m}{it}

\newtheorem{theorem}{Theorem}[subsection]
\newtheorem{proposition}[theorem]{Proposition}
\newtheorem{corollary}[theorem]{Corollary}
\newtheorem{lemma}[theorem]{Lemma}
\newtheorem{conjecture}[theorem]{Conjecture}

\newtheorem*{conjecturestar}{Conjecture} 
\newtheorem*{theoremstar}{Theorem}

\theoremstyle{definition}
\newtheorem{definition}[theorem]{Definition}

\newtheorem{remark}[theorem]{Remark}

\crefname{proposition}{Proposition}{Propositions}
\crefname{lemma}{Lemma}{Lemmas}
\crefname{theorem}{Theorem}{Theorems}
\crefname{corollary}{Corollary}{Corollaries}
\crefname{remark}{Remark}{Remarks}
\crefname{definition}{Definition}{Definitions}
\crefname{section}{Section}{Sections}
\crefname{conjecture}{Conjecture}{Conjectures}

\mathchardef\mhyphen="2D
\title{On completed homology and a conjecture of Venkatesh}
\author{Douglas Molin}\thanks{Department of Mathematical Sciences, Chalmers University of Technology and the University of Gothenburg, 412 96 Gothenburg, Sweden. E-mail: molind@chalmers.se, ORCID: 0009-0005-5266-1577}

\address{Department of Mathematical Sciences, Chalmers University of Technology and the University of Gothenburg, 412 96 Gothenburg, Sweden}
\email{molind@chalmers.se}

\input xy
\xyoption{all}
\usepackage{amscd}

\begin{document}

\begin{abstract}
	Let $F$ be a CM field and $\Pi$ a regular algebraic cuspidal automorphic representation of $\mathbf{G}=\PGL_2/F$. A conjecture of Venkatesh describes the structure of the contribution of $\Pi$ to the homology of the locally symmetric spaces associated to $\mathbf{G}$. We investigate this conjecture in the setting of $p$-adic homology with $p$ a totally split prime, using completed homology, the Taylor--Wiles method and the $p$-adic local Langlands correspondence for $\operatorname{GL}_2(\mathbb{Q}_p)$. Our main result is a `big $R=T$' theorem in characteristic 0, from which we deduce a conditional proof of the $p$-adic realisation of Venkatesh's conjecture.
\end{abstract}
\maketitle

\tableofcontents

\counterwithin{equation}{subsection}

\section{Introduction}
Let $\G=\PGL_2/F$ where $F$ is a CM field of degree $2l_0$, fix a prime $p$ and suppose $\Pi$ is a regular algebraic cuspidal automorphic representation of $\G(\A_F)$. Then $\Pi$ contributes to the $p$-adic homology of the locally symmetric spaces of $\G$ in the following sense. If $U\subset\G(\A_F^\infty)$ is an open compact subgroup, there is an associated locally symmetric space $X_U$ which is a smooth manifold of dimension $3l_0$ equipped with a $p$-adic local system $\sigma$ determined by the weight $\lambda$ of $\Pi$. The homology group $H_\ast(X_U,\sigma)$ is a finite-dimensional $\Q_p$-vector space equipped with an action of Hecke operators, and we define its `$\Pi$-part' as the eigenspace
\[
	H_\ast(X_U, \sigma)_\Pi \subset H_\ast(X_U,\sigma)
\]
of homology classes with the same system of Hecke eigenvalues as $\Pi$. Let us assume for the purpose of this introduction that all eigenvalues of $\Pi$ are rational and that $\Pi$ is unramified at all places above $p$. Using $(\mathfrak{g},K)$-cohomology, one can prove that, if $m =\dim (\Pi^\infty)^U$, then the following dimension formula holds:
\[
	\dim_{\Q_p} H_{l_0+i}(X_U,\sigma)_\Pi = \begin{cases}
		m \cdot \binom{l_0}{i} &\text{ if }i=0,\dots,l_0\\
		0&\text{ otherwise.}
	\end{cases} \tag{$\star$}
\]
To give an arithmetic explanation of this `spreading out' in multiple degrees of Hecke eigenspaces, Venkatesh has made the following conjecture.
\begin{conjecturestar}(Venkatesh) With notation as above, there exists a $\Q_p$-vector space $V_\Pi$ of dimension $l_0$ and a natural action
\[
    \bigwedge^\ast V_\Pi \circlearrowright H_\ast(X_U,\sigma)_\Pi
\]
such that $H_\ast(X_U,\sigma)_\Pi$ is a free graded module of rank $m$.

\end{conjecturestar}
In fact, Venkatesh's conjecture concerns homology with $\Q$-coefficients and involves a motive conjecturally attached to $\Pi$, whereas this article concerns the $p$-adic realisation of the motivic conjecture, where the role of the motive is played by a Galois representation whose existence is known.

Let $S$ denote the set of places of $F$ where $\Pi$ ramifies and $\GFS$ the Galois group of the maximal extension of $F$ unramified outside of $S$. Using the construction in \cite{hltt_2016} or \cite{scholze_2015}, one can associate to $\Pi$ a Galois representation $\rho = \rho_{\Pi,\iota}\colon \GFS \to \GL_2(\Q_p)$ to $\Pi$ and its dual adjoint Bloch--Kato Selmer group (see \Cref{section_galoiscohomology} for a definition)
\[
	V_\Pi \coloneqq H^1_f(\GFS,\ad^0 \rho(1)).
\]
The Bloch--Kato conjecture predicts that
\[
    \dim V_\Pi = \ord_{s=1} L(s,\ad^0\Pi),
\]
and the right-hand side is known to be $l_0 = \frac{1}{2}[F\colon \Q]$, which is consistent with Venkatesh's conjecture and the dimension formula $(\star)$.

In this article, we define an action as in the conjecture, under various assumptions on $\rho$, as well as that $p\geq 5$ and $p$ is totally split in $F$. Before stating our main result, let us describe in broad terms the strategy of proof, which is deformation-theoretic. The basic idea is to place the $\Pi$-part in a $p$-adic family by relating it to completed homology and Galois representations. The former is a $p$-adic representation $\widetilde{H}$ of $G=\prod_{v\mid p} \G(F_v)$ equipped with an action of a `big' Hecke algebra $T$, and there is a spectral sequence connecting $\widetilde{H}$ to the $\Pi$-part. 

The construction of $\rho$ furnishes a surjective map $R \to T$ where $R$ is the unrestricted deformation ring of $\rho$; in this way, we think of the $\Pi$-part as a coherent sheaf on the deformation space of $\rho$. Its support is the space of deformations of $\rho$ preserving $p$-adic Hodge theoretic properties of the restrictions $\rho|_{\GFv}$ for $v\mid p$. Tautologically, one can describe this space as an intersection taking place within the unrestricted deformation space of the tuple $(\rho|_{\GFv})_{v\mid p}$, namely between the unrestricted deformation space of the global representation $\rho$ with the space of deformations determined by the aforementioned conditions. In this way, one can view the $\Pi$-part as a module over the $\Tor$-algebra representing the derived intersection, which one would like to identify with the exterior algebra $\wedge^\ast V_\Pi$. In particular, the support of the $\Pi$-part should be a point (equivalently, $\dim V_\Pi = l_0$), which is predicted to hold by the Bloch--Kato conjecture and known in many cases due to work of A'Campo \cite{acampo2024}.
\newcommand{\Rloclambda}{\Rloc^\lambda}

For the purpose of giving a rough statement of our main result, let $\Rloc$ be the deformation ring of $(\rho|_{\GFv})_{v\mid p}$ and $\Rloc \surj \Rloclambda$ the quotient representing geometric deformations of fixed Hodge--Tate weights equal to those of $\rho$. Finally, let $\Rgl$ denote the representing ring of deformations of $\rho$, which turns out to be a quotient of $\Rloc$ as well. The space of global deformations satisfying the local conditions determined by $\lambda$ is represented by the tensor product $\Rgl \otimes_{\Rloc} \Rloclambda$. We now present a simplified version of our main result.
\begin{theoremstar} (\Cref{maintheorem})\label[theorem]{maintheoremintro}
	Suppose $p\geq 5$ is totally split in $F$ and that $\rho$ has no global deformations satisfying the local conditions induced by $\lambda$. Then, under various additional assumptions on $\rho$, the graded $E$-vector space $H_\ast(X_{U},\sigma)_\Pi$ has a canonical structure of free module of rank $m_\Pi$ over the $\Tor$-algebra
   \[
	   \Tor_\ast^{\Rloc}\big(\Rgl,\Rloclambda\big).
   \]
   If, in addition, $\Rgl$ is formally smooth, there is a canonical isomorphism of graded rings
   \[
	\Tor_\ast^{\Rloc}\big(\Rgl,\Rloclambda\big) \cong {\bigwedge}^\ast \Hfd.
   \]
\end{theoremstar}
The proof of \Cref{maintheoremintro} relies on the Taylor--Wiles method and its adaptation to completed homology by Gee--Newton \cite{gee_newton_2022}. We prove a `big $R=T$' theorem in characteristic $0$, identifying the spectrum of $T$ with the unrestricted deformation space of $\rho$. The proof relies crucially on the condition that $p$ is totally split in $F$ and a type of local-global compatibility conjecture at places $v\mid p$ (\Cref{localglobal}). Under these assumptions, we are able to leverage the $p$-adic local Langlands correspondence for $\GL_2(\Q_p)$ as described in \cite{paskunas_2013}. Pa\v{s}k\={u}nas' theory provides an equivalence between a suitable category of $G$-representations in which we find $\widetilde{H}$ and a certain category of modules over $\Rloc$. Under this equivalence, the image of $\widetilde{H}$ is a finitely generated $\Rloc$-module, which allows the Taylor--Wiles method to be applied; we learned of this strategy in \cite{pan_2022}.

Proving the final part of the theorem amounts to leveraging the formal smoothness assumption to compute the $\Tor$-groups of the intersection $\Rgl \otimes_{\Rloc} \Rloclambda$. This can be done by studying a natural short exact sequence of Galois cohomology groups obtained from the Poitou--Tate sequence.

A similar conditional result for $\GL_n/\Q, n\geq 2$ has been proved using the theory of eigenvarieties and overconvergent cohomology by Hansen--Thorne \cite{hansen_thorne2017}, and our application of the Poitou--Tate sequence is based on an analogous argument therein. There is also the original result of Galatius--Venkatesh \cite{galatius_venkatesh2018} which is restricted to the Fontaine--Laffaille setting and conditional on the Calegari--Geraghty vanishing conjecture for mod~$p$ cohomology, which has been improved on in \cite{cai_tilouine2022,cai2021} and also \cite{molin2026}, which has mild assumptions but does not contain a proof of the naturality of the action. Our methods, while restricted to $n=2$ and $p$ totally split, place no restriction on the slope of $\Pi$ since we do not use overconvergent cohomology, are valid beyond the Fontaine--Laffaille range, and automatically establish naturality.

\subsection*{Outline}
Let us briefly outline the content of this article. In \Cref{sec_setup}, we describe our general setup, discussing the homology of locally symmetric spaces, Hecke operators, Galois cohomology and Bloch--Kato Selmer groups, and collect some lemmas from commutative and homological algebra for easy reference later. In \Cref{sec_galois} is devoted to Galois deformation theory and compiles results from the literature on Galois deformation rings which are used in the proof of the main theorem. In \Cref{sec_repr}, we turn to the representation-theoretic part of the story and discuss completed homology and the associated Galois representations. Finally, \Cref{SistaKapitlet} consists of the proof of our main result.

\subsection*{Acknowledgements}
I wish to thank Christian Johansson for suggesting this problem, sharing his knowledge and providing guidance throughout the whole project. I also thank Martin Raum for providing feedback on a previous version of this text. 

\subsection*{Statement of AI use} ChatGPT was used to improve the presentation and to assist with literature searches. The results and arguments are the author's own.

\section{Setup and preliminaries}\label[section]{sec_setup}
Let $p\geq 5$ be a prime and $F$ a CM field of degree $[F\colon \Q] = 2l_0$ in which $p$ is totally split. We denote by $\A^\infty_F$ the ring of finite adèles of $F$. Let $E/\Q_p$ be a finite extension with ring of integers $\O$, uniformiser $\varpi$ and residue field $k=\O/\varpi$. At times, we tacitly assume $E \subset \overline{\Q}_p$ is `large enough', e.g. to contain Hecke eigenvalues or the embeddings of $F$. Unless otherwise stated, completed tensor products are taken over $\O$. When $R$ is an $\O$-algebra and $\p\subset R$ is an ideal not containing $\varpi$, we use the same symbol to denote the ideal $\p$ and the ideal it generates inside $R[1/\varpi]$. 
\subsection{Arithmetic locally symmetric spaces}
We begin by recalling the construction of the locally symmetric spaces associated to $\G=\PGL_{2}/F$. A complete reference is \cite[§6.1-2]{khare_thorne_2017}. We introduce the following notation:
\begin{itemize}
	\item $\G=\PGL_2/F$,
	\item $G_\infty=\G(F\otimes_\Q \R)$,
	\item $K_\infty\subset G_\infty$ a maximal compact connected subgroup,
	\item $D_\infty = G_\infty/K_\infty$, the symmetric space of $\G(F\otimes_\Q \R)$,
	\item $l_0 = \rank G_\infty - \rank K_\infty = [F\colon\Q]/2$,
	\item $q_0 = \frac{1}{2}(\dim D_\infty-l_0) = l_0$.
\end{itemize}
The integer $l_0$ is called the \emph{defect} of $\G$. For groups other than $\PGL_2/F$, one typically has $q_0\neq l_0$, and for this reason we have opted to maintain the distinction.
\begin{definition}
	Let $U\subset \G(\A^\infty_F)$ be an open compact subgroup. The \emph{locally symmetric space of level $U$} is the double quotient
	\[
		X_{U} = \G(F) \backslash (D_\infty \times \G(\A^\infty_F)/U).
	\]
\end{definition} 
The space $X_U$ decomposes as a finite disjoint union of spaces of the form $\Gamma_i \backslash D_\infty$ where $\Gamma_i=\G(F)\cap g_i U g_i^{-1}$ for some $g\in \G(\A_F^\infty)$.

\begin{definition}
	An open compact subgroup $U\subset \G(\A^\infty_F)$ is \emph{good} if $U=\prod_v U_v$ such that:
	\begin{itemize}
		\item[(i)] For every $v$, $U_v \subseteq \PGL_2(\O_{F_v})$.
		\item[(ii)] For every $g\in \G(\A^\infty_F)$ and every $h\in gUg^{-1} \cap \G(F),$ the eigenvalues of $h$ generate a torsion-free subgroup of $\overline{F}$ (i.e. the subgroup $gUg^{-1}$ is `neat').
	\end{itemize}
\end{definition}

\begin{proposition}\label[proposition]{XUresultsBS}\cite[Lem.~6.1]{khare_thorne_2017} 
	Let $U\subset \G(\A^\infty_F)$ be a good subgroup. Then $X_U$ is a smooth manifold of dimension $3l_0$ and homotopy equivalent to the geometric realisation of a finite simplicial complex. Moreover, if $U'\subset U$ is a normal compact open subgroup, then $U'$ is also good and $X_{U'}\to X_U$ is a Galois cover of smooth manifolds with Galois group $U/U'$.
\end{proposition}

\subsection{Homology of arithmetic locally symmetric spaces}\label[section]{ssec_homology}
Throughout this section, $U=U_pU^p\subset \G(\A^\infty_F)$ is a good subgroup. Here, $U_p=\prod_{v\mid p} U_v$ and $U^p=\prod_{v\nmid p} U_v$. We fix also a discrete left $\Z[U_p]$-module $M$, viewed as a $\Z[U]$-module by letting $U$ act via the projection $U\to U_p$.
\begin{definition}
	The \emph{local system associated to $M$} is the sheaf $\mathcal{L}_M$ of continuous sections of the map
	\[
		\G(F)\backslash(D_\infty \times \G(\A^\infty_F)/U) \times M/U \to \G(F)\backslash(D_\infty \times \G(\A^\infty_F)/U),
	\]
	We denote by $H_\ast (X_{U},\mathcal{L}_M)$ the homology of $X_U$ with local coefficients $\mathcal{L}_M$.
\end{definition}
There are two complexes commonly used to compute $H_\ast (X_{U},\mathcal{L}_M)$.
\begin{definition}
	The \emph{adèlic complex} of $X_U$ with coefficients in $M$ is the chain complex
	\[
		\Cad(U,M) = \operatorname{Sing}_\bullet\big(G_\infty/K_\infty \times \G(\A^\infty_F)\big) \otimes_{\Z[\G(F)\times U]} M,
	\]
	where $\operatorname{Sing}_\bullet\big(G_\infty/K_\infty \times \G(\A^\infty_F)\big)$ denotes the complex of singular chains with $\Z$-coefficients, viewed as a complex of right $\Z[\G(F)\times U]$-modules.
\end{definition}
\newcommand{\DBS}{D_\infty^{\operatorname{BS}}}
\newcommand{\XBS}{X^{\operatorname{BS}}_{U_pU^p}}
To define the Borel--Serre complex, consider the Borel--Serre bordification $D_\infty \subset D_\infty^{\operatorname{BS}}$ (see \cite[§7.1, Prop.~7.6]{borel_serre}) and the principal $U_p$-bundle
\[
	\G(F) \backslash (\DBS \times \G(\A^\infty_F)/\mathbf{1}_pU^p) \surj \G(F) \backslash (\DBS \times \G(\A^\infty_F)/U_pU^p) = \XBS,
\]
where $\G(\A^\infty_F)$ is equipped with the discrete topology. Fix a finite triangulation of $\XBS$, and consider its associated complex of simplicial chains with $\Z$-coefficients. Pulling back the triangulation via the map above we obtain -- possibly after a finite refinement -- a $K_p$-equivariant triangulation on the `infinite $p$-level' space $\G(F) \backslash (\DBS \times \G(\A^\infty_F)/\mathbf{1}_pU^p)$. The associated complex of simplicial chains is then a bounded complex of finitely generated and free $\Z[U_p]$-modules which we denote by $\CBS(U,\Z[U_p])$.

\begin{definition}
	The \emph{Borel--Serre complex of $X_U$ with $M$-coefficients} is the complex
	\[
		\CBS(U,\Z[U_p]) \otimes_{\Z[U_p]} M.
	\]
\end{definition}

\begin{proposition}\label[proposition]{chain_homotopic}
	The complexes $\Cad(U,M)$ and $\CBS(U,M)$ are chain homotopic as complexes of $\Z[U_p]$-modules, and
	\[
		H_\ast(\Cad(U,M)) \cong H_\ast(\CBS(U,M)) \cong H_\ast(X_U,\mathcal{L}_M).
	\]
\end{proposition}
We fix once and for all a chain homotopy equivalence $\Cad(U,M) \stackrel{\sim}{\to} \CBS(U,M)$, and define
\[
	H_\ast(X_U,M) \coloneqq H_\ast(X_U,\mathcal{L}_M).
\]
Aside from the trivial $\Z[U_p]$-modules $\O/\varpi^s$, the modules whose associated local systems we will use consider are indexed by $v$-adic Hodge types.
\begin{definition}\label[definition]{vadichodgetype}
	Let $v\in S_p$.	A $v$-adic Hodge type is a triple $(\lambda_v,\tau_v,\chi_v)$ where
	\begin{itemize}
		\item $\lambda_v=(a_v,b_v)\in \Z^2$ such that $b>a$.
		\item $\tau_v\colon \IFv \to \GL_2(E)$ is a representation of the inertia group with open kernel.
		\item $\chi_v\colon \GFv\to \O^\times$ is a continuous character such that $\chi_v|_{\IFv} = \eps \det\tau$.
	\end{itemize}
\end{definition}

\begin{theorem}	(\cite[A.1.5.]{breuil_mezard_henniart_2002}) Let $\tau_v\colon \IFv\to \GL_2(E)$ be a representation with open kernel and $K_v = \G(\O_{F_v})$. There exists a unique (up to isomorphism) smooth irreducible $K_v$-representation $\stv$ on an $E$-vector space characterised by the property
	\[
		\Hom_{K_v}(\Pi_v,\stv) \neq 0 \iff \operatorname{LL}(\Pi_v)|_{\IFv} \cong \tau_v, 
	\]
	when $\Pi_v$ ranges over all smooth absolutely irreducible infinite-dimensional $\G(F_v)$-representations over $E$ and $\operatorname{LL}(\Pi_v)$ is the Weil--Deligne representation associated to $\Pi_v$ by the classical local Langlands correspondence, normalised as in \cite{breuil_mezard_henniart_2002}.
\end{theorem}
To any $\lambda_v=(a_v,b_v)$ with $b>a$ and $a+b=1$, we define $\sigma(\lambda_v) \coloneqq (\Sym^{b_v-a_v-1} E^2) \otimes (\det)^{a_v}$.
\begin{definition}Let $\wtchiv$ be a $v$-adic Hodge type. The $K_v$-representation associated to $\wtchiv$ is the tensor product
	\[
		\swtv \coloneqq \stv \otimes \sigma(\lambda_v).
	\]
\end{definition}
Since $K_v$ is compact, there is a $K_v$-stable $\O$-lattice
\[
	\sigma^{\circ}(\lambda_v,\tau_v) \subset \swtv.
\]
Given a tuple $\wtchi=\wtchiv_{v\in S_p}$ of $v$-adic Hodge types, we define a $K$-representation $\swtp$ with $\O$-lattice $\sowtp$ by forming the tensor product over $v\in S_p$:
\[
	\sowtp \coloneqq \bigotimes_{v\mid p}\sigma^\circ(\lambda_v,\tau_v) \subset \swtp \coloneqq \bigotimes_{v\mid p} \swtv.
\]
\subsection{Hecke operators}\label[section]{sec_hecke}
Let $S$ be a finite set of finite places of $F$ containing the set $S_p$ of places above $p$, and suppose $K\subset \G(\A^\infty_F)$ is a good open compact subgroup which is hyperspecial outside $S$. We write $K^S = \prod_{v\notin S} K_v$. The \emph{abstract (spherical) Hecke algebra} is the convolution algebra
\[
    \T^S \coloneqq \mathcal{H}(\G(\A^{S}_F),K^S)
\]
of $K^S$-biinvariant functions $\G(\A^S_F)\to \O$. Given an automorphic representation $\Pi$ unramified outside $S$ and an isomorphism $\iota\colon\overline{\Q}_p\cong \C$, we have a map 
\[
	\T^S \to \End_\C \left({\bigotimes}'_{v\notin S} \Pi_v^{K_v}\right)
\]
with kernel a maximal ideal which we denote $\mathfrak{N}_{\Pi,\iota}$. The homology groups of $X_U$ are also equipped with an action of Hecke operators, meaning there is a map
\[
    \T^S \to \End_\O\big(H_\ast(X_U,M)\big).
\]
We define the $\Pi$-part of $H_\ast(X_U,M)$ as the localisation
\[
    H_\ast(X_U,M)_\Pi \coloneqq H_\ast(X_U,M)_{\mathfrak{N}_{\Pi,\iota}}.
\]
The dimensions of the graded pieces are given by the following formula.
\begin{theorem}\cite[Thm.~2.4.10]{10authorpaper}\label[theorem]{char0q0l0} Let $\Pi$ be a cuspidal representation of $\PGL_2(\A^\infty_F)$ of weight $\lambda=(a,b)$ (where $b>a$ and $a+b=1$), and set $m_\Pi =\dim_\C (\Pi^\infty)^{U^p}$. Then
	\[
		\dim_L H_i(X_{KU^p},\sigma(\lambda))_\p = \begin{cases}
			m_\Pi \cdot\binom{l_0}{i}&\text{ if }i\in [q_0,q_0+l_0],\\
			0&\text{ otherwise.}
		\end{cases}
	\] 
\end{theorem}
\subsection{Galois cohomology}\label[section]{section_galoiscohomology}
In this section, we prove a couple of results about Galois cohomology which will be of use later. For a complete reference, see \cite{milne}, and for an introduction see \cite{bellaiche}.

Let $S$ be a finite set of places of $F$ containing $S_p$ and suppose $\rho\colon\GFS\to \GL_2(E)$ is a continuous representation which is de Rham at all places above $p$. Let $v\in S$. If $v \mid p$, then
\[
	H^1_f(\GFv,\ad^0\rho) \coloneqq \ker\big(H^1(\GFv,\ad^0\rho) \to H^1(\GFv,\ad^0\rho\otimes_{\Q_p}B_{\operatorname{cris}})\big),
\]
where the map on the right is the natural one induced by the tensor product and $B_{\operatorname{cris}}$ is Fontaine's crystalline period ring. If $v\nmid p$, let $\IFv\subset \GFv$ denote the inertia subgroup and set
\[
	H^1_f(\GFv,\ad^0\rho) \coloneqq \ker\big(H^1(\GFv,\ad^0\rho) \to H^1(\IFv,\ad^0\rho)\big).
\]
where the map on the right is the one induced by the inclusion $\IFv\subset \GFv$. The \textit{geometric Bloch--Kato Selmer group} is the space of global classes with restrictions in the groups defined above, i.e.
\[
	H^1_f(\GFS,\ad^0\rho) \coloneqq \ker\left( H^1(\GFS,\ad^0\rho) \to \prod_{v\in S} \frac{H^1(\GFv,\ad^0\rho)}{H^1_f(\GFv,\ad^0\rho)}  \right).
\]
We use lowercase $h$ to denote dimensions of Galois cohomology groups and Selmer groups, e.g.
\[
	h^1_f(\GFS,V) = \dim_E H^1_f(\GFS,V).
\]

\begin{proposition}\label[proposition]{prop_generic_galois_cohomology}
If $v\in S\setminus S_p$ and the Weil--Deligne representation $\operatorname{WD}(\rho_v)$ is generic, then
\[
	h^0(\GFv,\ad^0\rho) = h^1(\GFv,\ad^0\rho).
\]
\end{proposition}
\begin{proof}
	The formula for the local Euler--Poincaré characteristic reads \[h^2(\GFv,\ad^0\rho) = h^1(\GFv,\ad^0\rho) - h^0(\GFv,\ad^0\rho).\] By \cite[Lem.~1.1.5]{allen2016}, the $h^2$-term vanishes assuming genericity of $\operatorname{WD}(\rho_v)$. 
\end{proof}

\begin{proposition}\label[proposition]{complementarydimension} Let $F$ be a CM field of degree $2l_0$ in which $p$ splits, $S$ a finite set of places of $F$ and  $\rho\colon\GFS\to \GL_2(E)$ an irreducible representation. Suppose that $\rho$ is de Rham with distinct Hodge--Tate weights at all $v\in S_p$, and moreover that $\operatorname{WD}(\rho_v)$ is generic for all $v\in S$. Then
\[
	h^1_f(\GFS,\ad^0\rho) = h^1_f(\GFS,\ad^0\rho(1))-l_0.
\]
\end{proposition}
\begin{proof}
By the Greenberg--Wiles formula, and the irreducibility of $\rho$, we have
\begin{equation*}
	\begin{split}
		&h^1_f(\GFS,\ad^0\rho)-h^1_f(\GFS,\ad^0\rho(1)) = \\ \sum_{v\in S} \big(&h^1_f(\GFv,\ad^0\rho)-h^0(\GFv,\ad^0\rho) \big) - \sum_{v\mid \infty} h^0(\GFv,\ad^0\rho).\\
	\end{split}
\end{equation*}
By \Cref{prop_generic_galois_cohomology}, the terms coming from $v \in S\setminus S_p$ vanish. Since $F$ is totally complex, the expression simplifies to
\[
	\sum_{v\mid p} \big(h^1_f(\GFv,\ad^0\rho)-h^0(\GFv,\ad^0\rho) \big) - \sum_{v\mid \infty}\dim \ad^0\rho.
\]
Now, $p$ is totally split in $F$ and the sum over $v\mid p$ counts the total multiplicities of the negative Hodge--Tate weights of the $\ad^0\rho|_{\GFv}$ (see \cite[Cor.~3.8.4]{bloch_kato1990}). Since the Hodge--Tate weights of $\rho_v$ are assumed to be distinct, the total expression therefore equals $-[F\colon \Q]/2 = -l_0$.
\end{proof}

\begin{theorem}\label[theorem]{tangentspacesequence}
	Let $F$ be a CM field, $S$ a finite set of places of $F$ and $\rho\colon\GFS\to \GL_2(E)$ a representation which is de Rham at all places above $p$. Then there is an exact sequence
	\begin{align*}
		0 &\to \Hf \to \H \to \prod_{v\in S}\frac{\Hp}{\Hfp}\\
		  &\to \Hfd^\vee \to H^2(\GFS,\ad^0\rho).
	\end{align*}
	Moreover, if the Weil--Deligne representation $\operatorname{WD}(\rho_v)$ associated to $\rho_v$ is generic for every $v\in S\setminus S_p$, the corresponding factors in the third term vanish.
\end{theorem}
\begin{proof}
	In \cite{washington_2000}, the analogous exact sequence with finite coefficients is derived from the Poitou--Tate sequence. We argue in the same way using the Poitou--Tate sequence for cohomology in characteristic $0$, which is readily obtained from \cite[Prop.~10]{washington_2000} by identifying, for any $E$-vector space $V$ with a continuous action of $\GFS$,
	\[
		H^1(\GFS,V) \cong \big(\varprojlim_{s} H^1(\GFS,V^\circ/\varpi^s)\big)\otimes_\O E,
	\]
	where $V^\circ \subset V$ is an arbitrary choice of $\GFS$-stable $\O$-lattice (recall that inverse limits are exact on compact modules). To prove the second part, let $V=\ad^0\rho$ and note (\Cref{complementarydimension}) that the genericity assumption at $v\in S\setminus S_p$ implies
	\[
		h^0(\GFv,V) = h^1(\GFv,V).
	\]
	Thus, it suffices to prove that the left-hand side equals $h^1_{\operatorname{ur}}(\GFv,V)$. We follow \cite[Prop.~2.3(a)]{bellaiche}. The inflation-restriction sequence yields an isomorphism
	\[
		H^1(\GFv/\IFv,V^{\IFv}) \cong H^1_{f}(\GFv,V).
	\]
	Since $\GFv/\IFv \cong \widehat{\Z}$ has cohomological dimension $1$ (\cite[Prop.~6.1.9]{gille_szamuely2017}), the Euler--Poincaré characteristic formula implies 
	\[
	h^1(\GFv/\IFv,V^{\IFv}) = h^0(\GFv/\IFv,V^{\IFv}) = h^0(\GFv,V).
	\] 
	Hence $h^1_{f}(\GFv,V) = h^0(\GFv,V)$, as claimed.
\end{proof}

\subsection{Some commutative and homological algebra}
In this section we collect a few basic results from commutative and homological algebra for easy reference in later sections. The reader is advised to skip this section and refer back as needed.
Throughout this section, we let $(A,\m)$ be a Noetherian local ring, $I\subset A$ a proper ideal and $M$ a finitely generated $A$-module. We begin by recall some basic basic facts concerning the depth of $M$, following \cite[§18]{eisenbud_1995}.

An \emph{$M$-regular $I$-sequence} is a sequence $\astreck=a_1,\dots,a_r\in I$ such that for $i=1,\dots,r$, multiplication by $a_i$ is injective on $M/(a_1,\dots,a_{i-1})$. The \emph{$I$-depth} $\dpt_I(M)$ of $M$ equals the supremum of the lengths of $M$-regular $I$-sequences, or equivalently (see \cite[Prop.~18.4]{eisenbud_1995})
\[
	\dpt_I(M) = \min \{i\in \N\mid  \Ext^i_A (A/I, M) \neq 0\}.
\]
We refer to $A$-regular $\m$-sequences simply as \emph{regular sequences}. When $I=\m$, we write $\dpt_A(M)$ for $\dpt_{\m}(M)$. Clearly, $\dpt_I M \leq \dpt_A M.$
\begin{definition}
	Let $\astreck=a_1,\dots,a_r$ be a sequence of elements of $\m$. The Koszul complex of $\astreck$ is the complex $\Kbulleta$ of $A$-modules which is non-zero only in degrees $[0,r]$ where it is given by
\begin{align*} 
	K_n(\astreck) &= \bigoplus_{1\leq j_1 < \dots < j_n \leq r} Ae_{j_1,\dots,j_n}\\
	d(e_{j_1,\dots,j_n}) &= \sum_{i=1}^r (-1)^{i+1}a_i e_{j_1,\dots,j_{i-1},j_{i+1},\dots,j_n}.
\end{align*}
\end{definition}
The Koszul complex is exact precisely when the sequence $\astreck$ is regular on $A$, in which case $\Kbulleta$ is a free resolution of $A/(\astreck)$. More generally, one has the following homological characterisation of $M$-regular sequences.
\begin{proposition}\label[proposition]{prop_homologicalcriterion}
	Let $\astreck$ be an arbitrary sequence in $A$. Then $\astreck$ is $M$-regular if and only if
	\[
		H_i (\Kbulleta\otimes_A M) = 0 \text{ for all }i\geq 1.
	\]
\end{proposition}
Note that if $\astreck$ is $A$-regular, then $H_\ast (\Kbulleta\otimes_A M)\cong \Tor_\ast(A/(\astreck),M)$. Restated slightly, we have:
\begin{lemma}\label[lemma]{TorAmodI}
	Let $M,N$ be finitely generated modules over $A$. Suppose $I\subset \Ann N$ is an ideal generated by an $A$-regular and $M$-regular sequence. Then we have a natural isomorphism of complexes
	\[
		M \otimesL_A N \cong M/IM \otimesL_{A/I} N
	\]
\end{lemma}
\begin{proof}
	By \Cref{prop_homologicalcriterion},
	$M \otimesL_A N \cong M \otimesL_A A/I \otimesL_{A/I} N
			 \cong M/IM \otimesL_{A/I} N$.
\end{proof}
\begin{lemma}\label[lemma]{depth-l} 
	Suppose $I\subset A$ can be generated by a regular sequence. Then
	\[
		\dpt_I M + \dpt^\ast_I M = \dpt_I A ,
	\]
	where $\dpt^\ast_I M=\max\{i \mid \Tor_i^A(A/I,M)\neq 0 \}$.
\end{lemma}
\begin{proof}
	Let $r=\dpt_I A$ and $\astreck = a_1,\dots,a_r$ a regular sequence generating $I$. By \cite[Prop.~17.15]{eisenbud_1995}, there is an equivalence of chain complexes
	\[
		\Kbulleta \otimes_A M \cong \Hom_A(K_{r-\bullet}(\astreck), M)
	\]
	and hence for any $i$ an isomorphism
	$
		\Tor_i^A(A/I,M) \cong \Ext^{r-i}_A(A/I,M).
	$
	The statement follows.
\end{proof}
\begin{lemma}\label[lemma]{depthcompletion}
	Let $\Ahat$ and $\widehat{M}$ denote the $\m$-adic completions of $A$ and $M$. Then 
	$
	\dpt_A M = \dpt_{\Ahat}\widehat{M}.
	$
\end{lemma}
\begin{proof}
	Since $A$ is Noetherian, $\Ahat$ is a faithfully flat $A$-module and the maximal ideal of $\Ahat$ is $\m \Ahat$ \cite[Thm.~1.3.16]{liu_2010}. By flat base change \cite[Prop.~3.3.10]{weibel_2011}, we have an isomorphism of $A$-modules
	\[
	\Ext_{\Ahat}^i(\Ahat/\m \Ahat, \widehat{M}) \cong \Ext_A^i(A/\m, M)\otimes_A \Ahat.
	\]
	Since $\widehat{A}$ is faithfully flat, the right hand side vanishes if and only if $\Ext^i_{A}(A/\m, M)$ does.
\end{proof}
The depth of $M$ relates to the dimension of its support via the inequalities
\[
	\dpt_A M \leq \dim_A M \leq \dim A.
\]
If the first (resp. both) inequality is an equality, $M$ is called Cohen--Macaulay (resp. maximal Cohen--Macaulay).
\begin{lemma}\label[lemma]{maxCMmeansfree}(\cite[Tag 00NT]{stacks-project})
	Suppose $A$ is regular and that $M$ is maximal Cohen--Macaulay over $A$. Then $M$ is free.
\end{lemma}
\begin{lemma}\label[lemma]{completionismaxCM}
	Let $A$ be a complete local Noetherian $\O$-algebra and $M$ a finitely generated maximal Cohen--Macaulay $A$-module. Suppose $\p\subset \Spec A[1/\varpi]$ is a regular closed point. Then $\Mphat$ is free over $\Aphat$.
\end{lemma}
\begin{proof}
	We prove that $\Mphat$ is maximal Cohen--Macaulay over $\Aphat$ and use \Cref{maxCMmeansfree}. By \Cref{depthcompletion} and the fact that $\dim \Ap = \dim \Aphat$, we have
	\[
		 \dpt_{\Ap} \Mp = \dpt_{\Aphat} \Mphat\leq \dim \Aphat = \dim \Ap,
	\]
	and so it suffices to prove $\dim \Ap \leq \dpt_{\Ap} \Mp$. Now, $A$ is a quotient of a ring of formal power series over $\O$ by the Cohen structure theorem and thus $A/\p$ is a finite extension of $\O$. In particular, $\dim A/\p = 1$ and 
	\[
		\dim A_\p \leq \dim A - 1 = \dpt_A M - 1,
	\]
	where we in the equality have used the maximal Cohen--Macaulay property of $M$. Moving on, since $\m = \p + (\varpi)$, an application of \cite[Lem.~18.3]{eisenbud_1995} gives the inequality
	\[
		\dpt_{A} M - 1= \dpt_{\p + (\varpi)} M - 1 \leq \dpt_\p M.
	\]
	We have
	$
		\dpt_\p M \leq \dpt_{\Ap} M_\p
	$
	since the $\Ap$-depth can be calculated in terms of the group
	\[
		\Ext^\ast_{\Ap}(A_\p/\p,M_\p) \cong \Ext^\ast_A(A/\p,M)\otimes_A \Ap.
	\]
	Thus we have established $\dim A_\p \leq \dpt_{\Ap} \Mp$, and the theorem follows.
\end{proof}
\section{Galois deformation theory}\label[section]{sec_galois}
In this section, we discuss the deformation theory of Galois representations, giving definitions and citing results from the literature that we shall need in the sequel. following \cite[§4]{khare_thorne_2017}.

Let $\Artk$ be the category of discrete Artinian local $\O$-algebras with residue field $k$, and $\CNLk$ the category of complete Noetherian local $\O$-algebras with residue field $k$, with morphisms the continuous $\O$-algebra homomorphisms. The latter is equivalent to a full subcategory of the category $\operatorname{Pro}(\Artk)$, and contains $\Artk$ as a full subcategory.

Let $S$ be a finite set of finite places of $F$ and $\Gamma$ be one of the groups $\GFS$ and $\GFv$. Suppose $\rhobar\colon \Gamma \to \GL_2(k)$ is a continuous representation and fix a continuous character $\chi\colon\Gamma \to \O^\times$ such that $\chi \equiv \det \rhobar \ \mod \varpi.$ For $A\in \Artk$, let $\operatorname{red}_A$ denote the reduction map $\GL_2(A)\to \GL_2(k)$, and $\chi_A$ the composition $\Gamma\stackrel{\chi}{\to}\O^{\times}\to A^\times$. A \textit{framed deformation} of $\rhobar$ to $A$ with determinant $\chi$ is a continuous homomorphism $\rho_A\colon\Gamma\to \GL_2(A)$ such that $\det \rho_A = \chi_A$ and $\rhobar = \operatorname{red}_A \circ \rho_A$. Two framed deformations $\rho_A,\rho_A'$ are \textit{strictly equivalent} if there exists $\gamma \in \ker(\operatorname{red}_A)$ such that $\gamma \rho_A \gamma^{-1} = \rho_A'$. An equivalence class of framed deformations to $A$ under the relation of strict eruivalence is called a \textit{deformation} of $\rhobar$ to $A$. We have functors
\[
 	\Dbarchi, \Dbarchiv, \Dbarchivbox \colon \Artk \to \operatorname{Set},
\]
where for example $\Dbarchi(A)$ is the set of deformations of $\rhobar$ to $A$ of determinant $\chi$, and the superscript $(-)^\square$ means `framed'. These functors extend in a canonical way to $\CNLk$.

Deformations of a continuous characteristic $0$ representation $\rho\colon \Gamma \to \GL_2(E)$ are defined in the same way. Let $\ArtE $ and $\CNLE$ be the category of Artinian, resp. complete Noetherian, local $E$-algebras with residue field $E$. As before, we define functors
\[
	\Dchi, \Dchiv, \Dchivbox \colon  \CNLE \to \operatorname{Set}.
\]
Using Schlessinger's criterion as in \cite{mazur1989}, one proves that the framed deformation functors are representable; the same holds for the unframed counterparts if the representation in question is \textit{Schur}, i.e. has only scalar endomorphisms. Later on, $\chi$ will always equal the cyclotomic character $\eps$, hence we denote the representing objects (if they exist) of the deformation functors defined above simply by $R_{\rhobar},\Rrhobarv,\Rrhobarvbox$ and $\Rrho,\Rrhov,\Rrhovbox$.
\begin{definition}
A \textit{deformation problem} is a tuple
$
	\mathcal{S} = (\rhobar, \chi, S, \{\Dv\}_{v\in S})
$
where 
\begin{itemize}
	\item $\rhobar\colon \GFS\to \GL_2(k)$ is an absolutely irreducible representation,
	\item $\chi\colon \GFS\to \O^\times$ is a continuous character such that $\chi \equiv \det \rho \mod \varpi$,
	\item $S$ is a finite set of finite places of $F$ such that $\rhobar$ and $\chi$ are unramified outside $S$,
	\item For every $v\in S$, $\Dv$ is a representable closed subfunctor $\Dv\subseteq \Dbarchivbox$ such that $\Dv(A)$ is stable under conjugation by elements of $\ker( \operatorname{red}_A)$.
\end{itemize}
\end{definition}
\begin{definition}
	
A deformation $[\rho_A] \in \Dbarchi(A)$ is \textit{of type} $\mathcal{S}$ if $\rho_A$ is unramified outside $S$, satisfies $\det \rho_A = \chi$ and for every $v\in S$,  $\rho_A|_{\GFv} \in \Dv(A).$ We define a functor
\[
	\DS \colon \CNLk \to \operatorname{Set},
\]
by letting $\DS(A)\subseteq \Dbarchi(A)$ be the set of deformations of type $\mathcal{S}$.
\end{definition}

\begin{theorem}\label[theorem]{DSrepresentable}
	$\DS$ is represented by an object $\RS$ in $\CNLk$.
\end{theorem}
\begin{proof}
	This follows from Schlessinger's criterion, as in \cite{mazur1989}.
\end{proof}

We will also need the notion of pseudorepresentations of $\GFv$ and their deformations. If $A\in \CNLk$ with $2\in A^\times$, then an $A$-valued \textit{pseudorepresentation} of $\GFv$ is a continuous function $\psi \colon \GFv\to A$ such that: $\psi(\operatorname{id})=2$; for every $g_1,g_2\in G,\ \psi(g_1g_2)=\psi(g_2g_1)$;  for every $g_1,g_2,g_3\in G$, 
		\begin{equation*}
		\begin{split}
			\psi(g_1g_3g_2) &= \psi(g_1g_2)\psi(g_3) + \psi(g_1g_3)\psi(g_2) + \psi(g_2g_3)\psi(g_1)\\
			 			 &- \psi(g_1g_2g_3) - \psi(g_1)\psi(g_2)\psi(g_3).
		\end{split}
		\end{equation*}
The \textit{determinant}\footnote{Not to be confused with Chenevier's notion of \textit{determinants} \cite{chenevier2014}.} of a pseudorepresentation $\psi_v \colon \GFv \to A$ is the function $\psi_v \colon \GFv \to A$ defined by
	\[
		\det \psi_v(g) = \frac{1}{2}\big(\psi_v(g)^2 - \psi_v(g^2)\big).
	\]
We will consider $A$-valued pseudorepresentations where $A\in \CNLk$ or $\CNLE$. Given a representation $\rhobar_v \colon \GFv \to \GL_2(k)$, its trace $\tr \rhobar_v\colon \GFv \to k$ is a pseudorepresentation of determinant $\det \rhobar_v$, and for any $\chi_v\colon \GFv \to \O^\times$ such that $\chi_v \equiv \det \rhobar_v \mod \varpi$, we define a functor
\[
	\Dtrrhobarvchi \colon \Artk \to \operatorname{Set}
\]
mapping $A\in \Artk$ to the set of $A$-valued pseudorepresentations with determinant $\chi_v$ lifting $\tr\rhobar_v$.
\begin{theorem}
	The functor $\Dtrrhobarvchi$ is representable by an object $\Rpsv \in \CNLk$.
\end{theorem}
\begin{proof}
	See e.g. \cite[Prop.~2.3.1]{bockle2013}
\end{proof}
There are natural transformations
\newcommand{\Dchivboxrhobarv}{D^{\chi_v,\square}_{\rhobarv}}
\newcommand{\Dchivrhobarv}{D^{\chi_v}_{\rhobarv}}
$	\Dchivboxrhobarv \to \Dchivrhobarv \stackrel{\operatorname{tr}}{\to} \Dtrrhobarvchi,
$
and thus $\Rrhobarvbox$ is a $\Rpsv$-algebra. Hence, if $\Rrhobarv$ exists then we have homomorphisms 
$
	\Rpsv \to \Rrhobarv \to \Rrhobarvbox,
$
and $\Rrhobarvbox$ is formally smooth of relative dimension $3$ over $\Rrhobarv$.
\newcommand{\RvD}{R_{\rhobarv}^D}
For a deformation problem $\mathcal{S}=(\rhobar,S,\chi,\{D_v\}_{v\in S})$ with local conditions $D_v$ represented by $\Rrhobarvbox\surj \Rbarv$, we define
\[
	\RSlocS = \widehat{\bigotimes}_{v\in S} \Rbarv
\]
This ring can also be interpreted as the representing object of a deformation functor, as the following lemma shows. 
\begin{lemma}
	Let $T$ be a finite set, and suppose we have representable functors $D_v\colon\CNLk \to \operatorname{Set}$ for every $v\in T$. Denote by $R_v$ the representing object of $D_v$, and consider the functor $D_T=\prod_{v\in T} D_v\colon\CNLk \to \operatorname{Set}$. Then we have the following:
	\begin{itemize}
		\item[(i)] $D_T$ is represented by $R_T=\hatotimes_{v\in T} R_v$.
		\item[(ii)] If $\p_T\subset R_T[1/\varpi]$ is a closed point generated by the joint image of closed points $\p_v\subset R_v[1/\varpi]$, we have an isomorphism (in $\CNLE$)
		\[
			(R_T)_{\p_T}^\wedge \cong \widehat{\bigotimes}_{v\in T} (R_v)_{\p_v}^\wedge.
		\]
	\end{itemize}
\end{lemma}
\begin{proof}
	(i) Let $\m_v$ be the maximal ideal of $R_v$, and $A= \varprojlim_i A_i \in \CNLk \subset \operatorname{Pro}(\Artk)$. We have
	\[
		(\prod_{v\in T}D_v)(A) \cong \prod_{v\in T} \varprojlim_i \Hom_\O(R_v,A_i) = \prod_{v\in T} \varprojlim_i \varinjlim_{r} \Hom_{\O} (R_v/\m_v^r,A_i),
	\]
	and since filtered colimits commute with finite limits, the above is isomorphic to
	\[
		\varprojlim_i \varinjlim_{r} \Hom_\O(\bigotimes_{v\in T}R_v/\m_v^r,A) = \Hom_\O(\hatotimes_{v\in T} R_v,A).
	\]
	(ii) Following \cite[§2.3]{kisin2009}, the ring $(R_T)_{\p_T}^\wedge$ represents a functor
	\[
		(D_T)_{(\p_T)} \colon \CNLE \to \operatorname{Set},
	\]
	whose value at $B\in \CNLE$ is defined as a filtered colimit $(D_T)_{(\p_T)}(B) = \colim_A D_T(A)$ over a category of integral $\O$-subalgebras of $B$. Since filtered colimits commute with finite limits, one sees that
	$
		D_{(\p_T)} \cong \prod_{v\mid p} D_{(\p_v)}.
	$
	Using a similar argument as in (i), one proves (ii).
\end{proof}
The properties of local deformation rings at $v\in S\setminus S_p$ we will need are contained in the following two results.
\begin{theorem}\label[theorem]{dimensionlokaladefringar}\cite[Thm.~2.5]{shotton_2018}
	Let $v\in S\setminus S_p$. The local framed (fixed-determinant) deformation ring $\Rrhobarvbox$ is an equidimensional reduced complete intersection, flat of relative dimension $3$ over $\O$.
\end{theorem}
\begin{theorem}\label[theorem]{lokal_generisk_glatt}\cite[Prop.~1.2.2]{allen2016}
	Let $v\in S\setminus S_p$ and $\p_v\subset \Rrhobarvbox[1/\varpi]$ be the point corresponding to a representation $\rho_v\colon \GFv\to \GL_2(E)$. Then the Weil--Deligne representation $\operatorname{WD}(\rho_v)$ is generic if and only if $\p_v$ is a regular point of $\Spec \Rrhobarvbox$.
\end{theorem}
At the places above $p$, our local conditions will encode conditions from $p$-adic Hodge theory.
\begin{definition}
	Let $\wtchiv$ be a $v$-adic Hodge type (\Cref{vadichodgetype}). A representation
	\[
		\rho_v \colon \GFv \to \GL_2(E)
	\]
	is of type $\wtchiv$ if it is potentially semi-stable with Hodge--Tate weights equal to $\lambda_v$, has determinant equal to $\chi$, and $\operatorname{WD}(\rho_v)|_{\IFv} \cong \tau_v$, where $\operatorname{WD}(\rho_v)$ is the Weil--Deligne representation associated to $\rho_v$.
\end{definition}
The deformations of fixed $v$-adic Hodge type form a Zariski closed subset of the generic fibre of the local deformation ring, as the following theorem shows.
\newcommand{\Rvsvbox}{R_{\rhobar_v}^\square(\sigma_v)}
\newcommand{\Rvsv}{R_{\rhobar_v}(\sigma_v)}
\begin{theorem}\label[theorem]{pstringsregular}
	Let $v\in S_p$ and $\rhobar_v\colon \GFv\to \GL_2(k)$. Fix a $v$-adic Hodge type $\wtchiv$ with $\det\rhobar_v \equiv \chi$ and let $\sigma = \sigma(\lambda_v,\tau_v)$. Then we have the following:
	\begin{itemize}
		\item[(i)] \cite[Thm.~2.7.6]{kisin2008} There exists a conjugation-invariant, reduced, $\O$-flat quotient $\Rvsvbox$ of $\Rrhobarvbox$ such that for any closed point $\p_v\subset \Rrhobarvbox[1/\varpi]$, the corresponding representation is of type $\wtchiv$ if and only if $\p_v \in \Spec \Rvsvbox[1/\varpi]$.
		\item[(ii)] \cite[Thm.~D]{allen2016} Suppose $\rhobar_v$ is Schur and let $\Rvsv$ be the corresponding unframed deformation ring. Then, if $\p_v \subset \Rvsv[1/\varpi]$ corresponds to a representation $\rho_v$ such that $\operatorname{WD}(\rho_v)$ is generic, $\p_v$ is a regular point of $\Spec \Rvsvbox[1/\varpi]$ and the Zariski tangent space at $\p_v$ is given by the local Bloch--Kato Selmer group $H^1_f(\GFv,\ad^0\rho)$.
	\end{itemize}
\end{theorem}
The relation between the deformation ring of a characteristic $0$ representation and that of its mod~$p$ reduction is explained by the following two results.
\begin{theorem}\label[theorem]{lokalkar0karp}\cite[Lem.~2.3.3, Proposition 2.3.5]{kisin2009} 
	Let $\rhobar_v\colon \GFv \to \GL_2(k)$ and suppose that $\p_v\subset R_{\rhobar_v}^\square$ is a closed point corresponding to $\rho_v\colon \GFv\to \GL_2(E)$. Then the localisation and completion $(R_{\rhobar_v}^\square)_{\p_v}^\wedge$ represents $D_{\rho_v}^\square$. If $\rhobar_v$ is Schur, then $(\Rrhobarv)_{\p_v}^\wedge$ represents $D_{\rho_v}$. 
\end{theorem}
\begin{theorem}\label[theorem]{Skar0karp}
	Let $\mathcal{S}=(\rhobar, \chi, S, \{\Dbarchivbox\}_{v\in S})$ be the global deformation problem of deformations unramified outside of $S$, represented by $\RS$. Suppose $\p\subset \RS[1/\varpi]$ is a closed point corresponding to a representation $\rho \colon \GFS \to \GL_2(E)$. Then the localisation and completion $\RSrho\coloneqq\RSphat$ represents the functor $D_{\mathcal{S},\rho}\colon\CNLE\to \operatorname{Set}$ for which $D_{\mathcal{S},\rho}(A)$ is the set of deformations of $\rho$ to $A$ of determinant $\chi$ which are unramified outside of $\mathcal{S}$. Moreover, the tangent space of $\RSrho$ is given by
	\[
		\opT^0_E(\RSrho) \cong H^1(\GFS,\ad^0\rho).
	\]
\end{theorem}
\begin{proof}
	This is deduced from \Cref{pstringsregular}, see e.g. \cite[Thm.~5.1.7]{acampo2024}.
\end{proof}
When dealing with irreducible representations, passage from pseudodeformation rings to deformation rings is enabled by the following result.

\begin{theorem}\label[theorem]{completion_pseudo}
	Let $\rho_v \colon \GFv \to \GL_2(E)$ be an absolutely irreducible representation and let $\p_v$ be the closed point of $\Spec \Rtrrhobarv[1/\varpi]$ corresponding to $\tr \rho_v$. Then 	$(\Rtrrhobarv)_{\p_v}^\wedge \cong R_{\rho_v}.$
\end{theorem}
\begin{proof}
	The ring $(\Rtrrhobarv)_{\p_v}^\wedge$ is isomorphic to $R_{\tr \rho_v}$ by \cite[Lem.~2.3.3, Proposition 2.3.5]{kisin2008}. The latter is canonically isomorphic to $R_{\rho_v}$ since $\rho_v$ is absolutely irreducible -- a consequence of the main theorem of \cite{nyssen1996}.
\end{proof}
\section{Representations and completed homology} \label[section]{sec_repr}
\subsection{Iwasawa modules and categories of smooth representations}\label[section]{sec_iwasawa}
In this subsection, we recall some general facts on categories of smooth representations and modules over Iwasawa algebras, following \cite{emerton_2010}.
\begin{definition}
	The \emph{Iwasawa algebra} (with $\O$-coefficients) of a compact $p$-adic analytic group $K$ is the profinite (possibly non-commutative) ring
\[
	\OK = \varprojlim_{U_p} \O[K/U_p],
\]
where $U_p$ runs over normal subgroups of $K$.
\end{definition}
\begin{definition}
	Let $G$ be a $p$-adic analytic group, $K\subset G$ a compact subgroup and fix a continuous character $\zeta \colon Z(G) \to \O^\times$. The category $\ModpfaGzetaO$ of profinite augmented $\O[G]$-modules is the category of profinite topological $\O[[K]]$-modules with a compatible action of $G$, central character $\zeta$, and admitting a neighbourhood basis of the identity given by $\O[[K]]$-submodules. Morphisms are given by continuous $G$-equivariant $\OK$-module homomorphisms. The full subcategory of modules which are finitely generated over $\OK$ is denoted $\ModfgaGzetaO$. 
\end{definition}
The category $\ModpfaGzetaO$ abelian and independent of the choice of $K$. Pontryagin duality defines an anti-equivalence of categories
\begin{align*}
	\Mod^{\operatorname{sm}}_{G,\zeta}(\O) &\to \ModpfaGzetaO \\
	V &\mapsto V^\vee \coloneqq \Hom_\O(V,E/\O)
\end{align*}
where $\Mod^{\operatorname{sm}}_{G,\zeta}(\O)$ is the category of smooth $G$-representations over $\O$ with central character $\zeta$. Here, a $G$-representation $V$ is \emph{smooth} if $V = \cup_{K,s} V^K[\varpi^s]$. If, moreover, every $V^K[\varpi^s]$ is finitely generated over $\O$, we say $V$ is \emph{admissible} and if, for every $v\in V$, the subrepresentation generated by $v$ is admissible we say $V$ is \emph{locally admissible}. The full subcategory of $\Mod^{\operatorname{sm}}_{G,\zeta}(\O)$ consisting of locally admissible (resp. admissible) representations is denoted $\Mod^{\operatorname{ladm}}_{G,\zeta}(\O)$ (resp. $\Mod^{\operatorname{adm}}_{G,\zeta}(\O)$). These categories are also abelian, and if we define
$
	\CGzetaO \coloneqq (\Mod^{\operatorname{ladm}}_{G,\zeta})^\vee,
$
we have fully faithful embeddings and anti-equivalences
\[\begin{tikzcd}
	{\Mod^{\operatorname{adm}}_{G,\zeta}(\O)} & {\Mod^{\operatorname{ladm}}_{G,\zeta}(\O)} & {\Mod^{\operatorname{sm}}_{G,\zeta}(\O)} \\
	\ModfgaGzetaO & \CGzetaO & \ModpfaGzetaO.
	\arrow[hook, from=1-1, to=1-2]
	\arrow["\sim"', from=1-1, to=2-1]
	\arrow[hook, from=1-2, to=1-3]
	\arrow["\sim"', from=1-2, to=2-2]
	\arrow["\sim"', from=1-3, to=2-3]
	\arrow[hook, from=2-1, to=2-2]
	\arrow[hook, from=2-2, to=2-3]
\end{tikzcd}\]
\subsection{Representations of $\GL_2(\Q_p)$}\label[section]{GenerellPaskunas}
In this subsection, we cite some results on representations of $\GL_2(\Q_p)$ from \cite{paskunas_2013}. For now, we let $G=\GL_2(\Q_p)$, $B\subset G$ the upper-triangular matrices and $K=\GL_2(\Z_p)$. Fix a continuous character $\zeta\colon Z(G) \to \O^\times$. Later on, we will only consider the case when $\zeta$ is trivial.

The category of locally admissible representations is equivalent to a direct product of subcategories,
\[
	\ModladmGzetaO \simeq \prod_\B \ModladmGzetaO_\B
\]
(\cite[Prop.~5.34]{paskunas_2013}). Here, $\ModladmGzetaO_\B$ denotes the full subcategory consisting of representations with the property that all irreducible subquotients are isomorphic to one of a finite number of representations constituting the \textit{block} $\B$. Pontryagin duality preserves this decomposition, yielding a corresponding statement for the dual category, i.e.
\[
	\CGzetaO \simeq \prod_\B \CBO.
\]
Since $k$ is not algebraically closed, there may be irreducible representations which are not absolutely irreducible. After a finite base change, such a representation decomposes into a direct sum of absolutely irreducible representations. Therefore, we will restrict our attention to blocks $\B$ containing an absolutely irreducible representation and tacitly assume $k$ is large enough. Case (iii) below will not feature later on when we let $\zeta$ be the trivial character. Recall that $\eps$ denotes the cyclotomic character and $\overline{\eps}$ its $\mod \varpi$ reduction.

\begin{proposition}\label[proposition]{prop_gl2blocks}
	(\cite[Prop.~5.42]{paskunas_2013}) Suppose $p\geq5$. The blocks $\B$ containing an absolutely irreducible representation are as follows:
	\begin{enumerate}
		\item[(i)] $\B = \{\pi\}$, where $\pi$ is supersingular.
 		\item[(ii)] $\B = \{\Ind_B^G (\chi_1 \otimes \chi_2\bar{\eps}^{-1}),\ \Ind_B^G (\chi_2 \otimes \chi_1\bar{\eps}^{-1})\}$ with $\chi_1/\chi_2 \neq 1,\bar{\eps}^{\pm1}$.
 		\item[(iii)] $\B =\{\Ind_B^G(\chi\otimes \chi\bar{\eps}^{-1})\}$.
 		\item[(iv)] $\B=\{\eta, \operatorname{Sp}\otimes\eta,(\Ind_B^G\alpha)\otimes \eta\}$, where $\alpha(\begin{pmatrix}
			x & y \\ 0 & z
		\end{pmatrix})=xz^{-1}|xz^{-1}| \mod \varpi$ and $\eta \colon G\to k^\times$ is a smooth character.    
	\end{enumerate}
\end{proposition}
For every $\pi\in \B$, we let $P_\pi\surj \pi^\vee$ be a projective envelope in $\CGzetaO$ and define
\[
	\PB \coloneqq \bigoplus_{\pi\in \B} P_\pi,\quad \EB \coloneqq \End_{\CBO}(\PB),\quad \ZB \coloneqq Z(\EB).
\]
Then $P_\B$ is a projective envelope of $\bigoplus_{\pi\in\B}\pi^\vee$ and is moreover a projective generator in $\CBO$. The possibly non-commutative ring $E_\B$ is compact with respect to a natural topology and there is an equivalence of abelian categories
\begin{align*}
	\CBO &\to \RModcptEB \\
	V &\mapsto \Hom_{\CBO}(P_\B, V).
\end{align*}
The inverse of the functor $\Hom_{\CBO}(P_\B,-)$ is given by the completed tensor product, so that for any $V\in \CBO$, there is a canonical isomorphism
\[
	V \cong \Hom_{\CBO}(P_\B, V) \otimeshat_{E_\B} P_\B.
\]
We will now relate $\EB$ to a pseudodeformation ring. To every block $\B$ in \Cref{prop_gl2blocks} we attach a semisimple 2-dimensional $k$-representation $\rhobar_\B$ of $\GQp$ with determinant $\det \rhobar_\B = \zeta\bar{\eps}$ as follows (using the same numbering):
\begin{enumerate}
	\item[(i)] $\rhobar_\B = \mathbf{V}(\pi)$ where $\mathbf{V}$ is Colmez' Montreal functor (see \cite[5.7]{paskunas_2013}).
 	\item[(ii)] $\rhobar_\B = \chi_1 \oplus \chi_2$ (viewing $\chi_1,\chi_2$ as characters via local class field theory).
 	\item[(iii)] $\rhobar_\B = \chi \oplus \chi$.
 	\item[(iv)] $\rhobar_\B = \eta_0 \oplus \eta_0\bar{\eps}$.
\end{enumerate}

\begin{theorem}\label[theorem]{EBresultA}
	Let $\B$ be a block containing an irreducible representation. Then we have the following:
	\begin{itemize}
		\item[(a)] $\EB$ is a finitely generated module over $\ZB$.
		\item[(b)] There is a canonical isomorphism of $\O$-algebras
		$
			\ZB \cong \RpsB.
		$ 
		\item[(c)] Let $\B$ be one of the blocks (i), (ii) or (iv), and suppose $\p_v\subset \RpsB[1/\varpi]$ is a closed point corresponding to an irreducible $\GQp$-representation $\rho_v$, so that $(\RpsB)_{\p_v}^\wedge \cong R_{\rho_v}$ (\Cref{completion_pseudo}). Then
		$
			(E_\B)_{\p_v}^\wedge \cong M_{|\B|}(\Rrhov),
		$
		a rank $|\B|$ matrix algebra with coefficients in $\Rrhov$.
	\end{itemize}	
\end{theorem}

\begin{proof}
	For statements (a) and (b), see \cite[Thm.~1.5, Corollary 8.11, Corollary 9.25, Lem.~10.90]{paskunas_2013}.
	(c) In the `supersingular block' (i), the natural map $\ZB \to \EB$ is an isomorphism (\cite[Prop.~6.3]{paskunas_2013}). The case (ii) is dealt with in \cite[Cor.~B.27]{paskunas_2013}, and we outline the proof here. We have $P = P_1 \oplus P_2$ and a presentation
	\[
		\EB = \big(\Hom_{\mathfrak{C}_{G,\zeta}(\O)}(P_i,P_j)\big)_{1\leq i,j\leq 2} \cong \begin{pmatrix}
			\ZB \mathbf{1} & \ZB b_1 \\ \ZB b_2 & \ZB \mathbf{1}
		\end{pmatrix}
	\]
	where $b_1 \circ b_2 = b_2 \circ b_1 = c \ZB$ for an element $c \in \ZB$ with the property that a point $\p_v\subset \RpsB[1/\varpi]$ defines an irreducible representation if and only if $c\notin \p_v$ (that is, the reducibility ideal is principal, generated by $c$). Consequently, if $\p_v$ is such a point then
	\[
		(\EB)_{\p_v}^\wedge \cong (\EB[1/c])_{\p_v}^\wedge \cong M_2(\ZBphat).
	\]
	The block (iv) is dealt with similarly. In this case, we have a $3$-by-$3$ presentation of $\EB$, where the relevant relations are generated by two elements $c_0,c_1\in \ZB$, see \cite[Lem.~10.91]{paskunas_2013}. As before, the point $\p_v$ defines an irreducible representation if $(c_0,c_1)\not\subset \p_v$. Since $c_i^{-1}(c_0,c_1)=\ZB[1/c_i]$, the same argument as in the previous case proves the claim.
\end{proof}
\subsection{Representations of $\prod{\PGL_2(\Q_p)}$}
Recall that we assume $p$ to be totally split in $F$. From now on, we let $G \coloneqq \prod_{v\mid p} G_v \coloneqq \prod_{v\mid p} \G(F_v)$ and $K \coloneqq \prod_{v\mid p} K_v \coloneqq \prod_{v\mid p} \G(\O_{F_v})$. Our category of interest is the Pontryagin dual of the category of locally admissible $G$-representations, equivalently
\[
	\CGO \simeq \big(\Mod^{\operatorname{ladm}}_{\prod_{v\mid p}\GL_2(\Q_p),1}(\O)\big)^\vee.
\]
The block decomposition results of \Cref{GenerellPaskunas} rely only on general facts about locally finite categories and still hold for $\CGO$. That is, there exists a set of blocks $\B$, projective generators $P_\B$ with endomorphism rings $\EB \coloneqq \End_{\CGO}(P_\B)$ with centres $\ZB \coloneqq Z(\EB)$ and equivalences of categories
\[
	\CGO \simeq \prod_\B \CBO \simeq \prod_\B \RModcptEB.
\]
\begin{lemma}\label[lemma]{OKändligEBändlig}
	Let $H \in \CBO$ and suppose $H$ is finitely generated as an $\O[[K]]$-module. Then $\Hom_{\CGO}(P_\B,H)$ is a finitely generated $E_\B$-module.
\end{lemma}
\begin{proof}
Recall from the discussion at the end of \Cref{sec_iwasawa} that $H$ being finitely generated over $\OK$ is equivalent to $H^\vee$ being admissible. Since $\EB=\End_{\CGO}(P_\B)$, it suffices to show that for some $r$, there is a surjection
$
	P_\B^{\oplus r} \surj H. 
$
Indeed, applying $\Hom_{\CGO}(P_\B,-)$ then proves the lemma. We claim that the cosocle $\operatorname{cosoc} H$ is a finite direct sum $\oplus_i \pi_i^\vee$ of irreducible objects. Indeed, every $\pi_i \hookrightarrow \oplus_i \pi_i = \operatorname{soc} H^\vee$ has non-zero $K_0$-invariants for any pro-$p$ group $K_0$. Since the dual $H^\vee$ is admissible, we see that $\operatorname{soc} H^\vee$ is a finite direct sum of irreducibles. The same then holds for $\operatorname{cosoc} H^\vee = (\operatorname{soc} H^{\vee})^\vee.$ Now, choose a surjection
$
	P_\B^{\oplus r} \surj \operatorname{cosoc} H
$
for some $r$. By the projectivity of $P_\B$, this map factors as
\[
	P_\B^{\oplus r} \to H \to \operatorname{cosoc} H,
\]
and the admissibility of $H^\vee$ implies that the second arrow is a superfluous surjection, by \cite[Lem.~4.6]{ceggps_2018}. Thus the first map is also surjective, which completes the proof.
\end{proof}
Pan \cite{pan_2022} has extended the results of the previous section to our setting. We define
\[
	\Rpsp  \coloneqq \widehat{\bigotimes}_{v\mid p} \Rpsv.
\]
\begin{proposition} \label[proposition]{EBresultB}
	Let $(\B_v)_{v\mid p}$ be blocks of $\CGvO$, each containing an absolutely irreducible representation, and define
	\[
		\B \coloneqq \bigotimes_{v\mid p} \B_v \coloneqq \{\bigotimes_{v\mid p} \pi_v \colon \pi_v \in \B_v\}.
	\]
Then $\B$ is a block of $\CGO$, and moreover:
	\begin{itemize}
		\item[(i)] $P_\B \coloneqq \widehat{\bigotimes}_{v\mid p} P_{\B_v}$ is a projective envelope of $\widehat{\bigotimes}_{v\mid p}\pi_v^\vee$.
		\item[(ii)] $\EB=\End_{\CGO}(P_\B)\cong \widehat{\bigotimes}_{v\mid p}E_{\B_v}$  
		\item[(iii)] $\EB$ is a finitely generated module over $\Rpsp$.
		\item[(iv)] Let $\p_v \subset \Rpsv[1/\varpi]$ be a point corresponding to the trace $\tr \rho_v$ of a irreducible representation $\rho_v$, and let $\p$ be the ideal of $\Rpsp=\widehat{\bigotimes}_{v\mid p}\Rpsv$ generated by the joint image of the $\p_v$'s. Then 
		\begin{equation*}
			\EBphat \cong M_{|\B|}(\Rpsphat).
		\end{equation*}
	\end{itemize}
\end{proposition}
\begin{proof}
	For (i-iii), see \cite[§3.4]{pan_2022}.	For (iv), the direct sum decomposition $P_{\B_v} = \bigoplus_{i_v=1}^{|\B_v|}P_{\B_v,i_v}$ yields a matrix presentation similar to the one in the proof of \Cref{EBresultA}, namely
	\[
		\EB = \End_{\CGO}\big(\widehat{\bigotimes_{v\mid p}} \bigoplus_{i_v=1}^{|\B_v|} P_{\B_v,i_v}\big) \cong \big(\widehat{\bigotimes}_{v\mid p} \Hom_{\CGvO}(P_{\B_v,i_v},P_{\B_v,j_v})\big)_{(i_v),(j_v)},
	\]
	where $(i_v)$ and $(j_v)$ run over sequences of length $|\B_v|$ such that $1\leq i_v, j_v \leq |\B_v|$. This defines a matrix presentation of $\EB$ over $\Rpsp$ as in the proof of Theorem 6.4(iv), and there exists an element $c \in \Rpsp$ such that $\EB[1/c]$ is a matrix algebra; indeed, we may take $c$ equal to the product of elements $c_v\in \Rpsv$ chosen such that $E_{\B_v}[1/c_v]$ is a matrix algebra.
\end{proof}

\subsection{Functors of twisted coinvariants}
Let $\sigmao = \sowtp = \otimes_{v\mid p} \sigma^\circ_v(\lambda_v,\tau_v)$ be an $\OK$-module of the form introduced in \Cref{ssec_homology}. Then, for any $M \in \RModcptOK$, we have a canonical isomorphism
\[
	N \hat{\otimes}_{\OK} \sigmao \cong N \otimes_{\OK} \sigmao,
\]
and $-\otimes_{\OK} \sigmao$ defines a right exact functor $\RMod^{\operatorname{cpt}}(\OK) \to \Mod^{\operatorname{cpt}}(\O)$ (\cite[Lem.~2.1]{brumer1966}). If $N=\hatotimes_{v\mid p} N_v$ where $N_v$ is an $\O[[K_v]]$-module for $v\mid p$, we have an isomorphism
\[
	N \otimes_{\OK} \sigmao \cong \widehat{\bigotimes}_{v\mid p} (N_v \otimes_{\O[[K_v]]} \sigma^\circ_v),
\]
We write $N(\sigmao) \coloneqq N \otimes_{\OK} \sigmao$ and $N(\sigma) \coloneqq N \otimes_{\OK} \sigma \cong N(\sigmao)[1/\varpi].$
\begin{proposition}\label[proposition]{Fsigmaisatensorproduct}\cite[Rem.~5.1.7]{gee_newton_2022}
	Let $\sigmao=\sowtp$ and suppose $P\in\RMod^{\operatorname{cpt}}(\OK)$ is projective. Then there is a natural isomorphism
	\[
		P(\sigmao) \cong \Hom_{\OK}^{\operatorname{cts}}(P,(\sigmao)^{\operatorname{d}})^{\operatorname{d}}
	\]
	where $(-)^d = \Hom_\O(-,\O)$ with the topology of pointwise convergence.
\end{proposition}
Let $\B=\otimes_{v\mid p} \B_v$ be a block of $G$ of the form considered in \Cref{EBresultB}. We now turn our attention to the module of twisted coinvariants $\PB(\sigma)=\PB \otimes_{\OK} \sigma$ of the projective generator $P_\B$ of $\CBO$, or rather its localisation and completion at a characteristic 0 point. Since $\EBphat$ is isomorphic to a matrix algebra over $\Rrhop$ when $\p$ is the ideal corresponding to a tuple $(\rho_v)$ of irreducible representations, $\PB(\sigma)_\p^\wedge$ is a direct sum of $|\B|$ pairwise isomorphic $\Rrhop$-modules. For this reason, it will suffice to consider a single summand of $P_\B$. For $v\mid p$, we let
\[
	\pi_v =\Ind_P^G (\chi_{v,1}\otimes \chi_{v,2}\eps^{-1})	\in \B_v,
\]
where $\chi_{v,1}/\chi_{v,2}\neq 1,\eps$ (allowing $\chi_{v,1}\eps = \chi_{v,2}$), and let $P \coloneqq \hatotimes_{v\mid p} P_{\pi_v} \surj \hatotimes_{v\mid p}\pi_v^\vee$ be the projective envelope as in \Cref{EBresultB}(a).
\newpage
\begin{proposition}\label[proposition]{PprojectiveoverK}
$P \in \RMod^{\operatorname{cpt}}(\OK)$ is projective.
\end{proposition}
\begin{proof}
	We mimic the proof of \cite[Lem.~B.8]{gee_newton_2022}, proceeding by induction on $|S_p|$. We have that $P=\widehat{\otimes}_{v\mid p} P_v$, where $P_v \in \RModcptOKv$. The case $|S_p|=1$ is \cite[Cor.~5.3]{paskunas_2015}. For the induction step, let $w\in S_p$ and define
	\[
		K^w = \prod_{v\in S_p\setminus\{w\}}K_v\quad\text{and}\quad P^w = \widehat{\otimes}_{v\in S_p\setminus \{w\}} P_v,
	\]
	so that $K=K_w\times K^w$ and $P=P_w \otimeshat P^w$. By the induction hypothesis, $P_w$ and $P^w$ are projective over $\O[[K_w]]$ and $\O[[K^w]]$, respectively. The universal property of the completed tensor product implies that for any compact $\OK$-module $N$,
	\[
		\Hom^{\operatorname{cts}}_{\O[[K_w\times K^w]]}(P_w\otimeshat_{\OK} P^w,N) \cong \Hom^{\operatorname{cts}}_{\O[[K_w]]}(P_w,\Hom^{\operatorname{cts}}_{\O[[K^w]]}(P^w,N)).
	\]
	Hence, the projectivity of $P$ follows from that of $P_w$ and $P^w$.
\end{proof}
\begin{theorem}\label[theorem]{Ppcoinvariants}
	Let $\wtchiv_{v\in S_p}$ be $v$-adic Hodge types and define $R_p(\sigma)=\hatotimes_{v\mid p} \Rvsv$. Then the action of $\Rpsp$ on $P(\sigmao)$ factors through $R_p(\sigma)$ and $P(\sigma)$ is locally free of rank $1$ over the regular locus of $\Spec R_p(\sigma)[1/\varpi]$.
\end{theorem}
\begin{proof}
	We have an isomorphism
	$
		P(\sigmao) \cong \widehat{\bigotimes}_{v\mid p} P_v(\sigma^\circ_v),
	$
	and the modules $P_v(\sigma^\circ_v)$ are $\O$-flat by \cite[Lem.~2.10]{paskunas_2015} and maximal Cohen--Macaulay over $\Rvsv$ by \cite[Cors.~6.4, 6.5]{paskunas_2015} and \Cref{Fsigmaisatensorproduct}. For every $v\mid p$, we fix a maximal regular sequence on $P_v(\sigma^\circ_v)$ of length $\dim \Rvsv$ containing $\varpi$ by extending the sequence $\{\varpi\}$ to a maximal regular sequence. For a flat compact $\O$-module $M$, the functor $-\otimeshat_\O M$ is exact and thus the union of the regular sequences is a regular sequence of length
	\[
		1 + \sum_{v\mid p} (\dim \Rvsv-1) = \dim R_p(\sigma).
	\]
	Hence, $\dpt_{R_p(\sigma)} P(\sigmao) = \dim R_p(\sigma)$, so $P(\sigmao)$ is maximal Cohen--Macaulay over $R_p(\sigma)$. Therefore, if $\p\subset R_p(\sigma)[1/\varpi]$ is a regular closed point, $P(\sigmao)_\p$ is free by \Cref{maxCMmeansfree}. Finally, to compute the rank, let $k(\p)=\otimes_{v\mid p} k(\p_v)\cong E$ be the residue field at $\p$. By \cite[Props.~2.22, 4.14]{paskunas_2015}, each $P_v(\sigma_v^\circ)_{\p_v}$ is of rank $1$ over $\Rvsv_{\p_v}$, and hence
	$
		\operatorname{rk}_{R_p(\sigma)_\p} \big(P(\sigmao)_\p\big) = \dim_{k(\p)}\big(k(\p)\otimes_{R_p(\sigma)_\p} P(\sigmao)_\p\big) = 1.
	$
\end{proof}
\begin{corollary}\label[corollary]{PBpcoinvariants}
	Let $\p\subset \Rpsp[1/\varpi]$ be a closed point corresponding to a tuple of absolutely irreducible representations, so that $\EBphat \cong M_{|\B|}(\Rrhop)$ by \Cref{EBresultB}. Then $\PB(\sigma)_\p^\wedge$ is isomorphic to $\Rrhop(\sigma)^{\oplus |\B|}$ as an $\EBphat$-module.
\end{corollary}
\begin{proof}
	This follows from \Cref{Ppcoinvariants} and the discussion at the beginning of this section. Indeed, 
	\[
		(\PB(\sigma))_\p^\wedge \cong \big(P(\sigmao)_\p^\wedge \big)^{\oplus |\B|} \cong \Rrhop(\sigma)^{\oplus |\B|},
	\]
	where the right-most module has the usual right action of $M_{|\B|}(\Rrhop)$.
\end{proof}

\subsection{Completed homology and the big Hecke algebra}
In this section, we recall the construction of Emerton's $p$-adically completed homology of the group $\mathbf{G}$ and the big Hecke algebra, and relate these to the finite-level homology groups of \Cref{ssec_homology}. For a survey, see \cite{calegari_emerton2012}.
\begin{definition}
	Let $U^p \subset \G(\A^{\infty,p}_F)$ be an open compact subgroup. The \emph{completed homology} with tame level $U^p$ is the projective limit
	\[
		\widetilde{H}_\ast (X_{U^p},\O) \coloneqq \varprojlim_{U_p} H_\ast(X_{U_pU^p},\O)
	\]
	where $U_p$ runs over a countable basis of neighbourhoods of the identity, normal in $K_1$ for any choice of $K_1\subset K$ such that $K_1U^p$ is good. 
\end{definition}
Since $H_\ast(X_{U_pU^p},\O)$ is an $\O[K/U_p]$-module, $\Htildestar$ is an $\O[[K]]$-module.

\begin{proposition}\label[proposition]{passagefinitelevel}
	Suppose $K_0\subset K$ is a subgroup such that $K_0U^p$ is good. Then there is a canonical isomorphism
	\[
		\Htildestar \cong H_\ast(X_{K_0U^p},\O[[K_0]]).
	\]
	Moreover, for any compact open $K_1\subset K$, $\Htildestar$ is a finitely generated $\O[[K_1]]$-module.
\end{proposition}
\begin{proof}
Let $K_0$ be as in the statement and consider the natural Hecke-equivariant map
\[
	\Cad(K_0, \O[[K]]) \to \varprojlim_{U_p} \Cad(K,\O[K/U_p]).
\]
Using our fixed choice of chain homotopy equivalence between the adèlic complex and the Borel--Serre complex, the map above corresponds to a map
\[
	\CBS(K,\O[[K]]) \to \varprojlim_{U_p} \CBS(K,\O[K/U_p])
\]
which is a quasi-isomorphism since $\CBS(K)$ consists of free and finitely generated $\Z[K]$-modules. Now,
\[
	\CBS(K,\O[K/U_p])\cong \CBS(U_p,\O)
\]
and inverse limits are exact on compact modules, so that
\[
	H_\ast\big(\varprojlim_{U_p} \CBS(U_p,\O)\big) \cong \varprojlim_{U_p} H_\ast(\CBS(U_p,\O)) \cong \Htilde_\ast(X_{U^p},\O),
\]
as required. Since $X_{K_0U^p}$ is homotopic to a finite simplicial complex \Cref{XUresultsBS}, the isomorphism above implies that $\Htildestar$ is finitely generated over $\O[[K_0]]$. Lastly, we note that if $K_1\subset K_0$ is normal, then $\O[[K_0]]$ is finitely generated over $\O[[K_1]]$.
\end{proof}

As explained in \cite[Rem.~3.4.13]{gee_newton_2022}, for any element $g\in G$ and open compact subgroup $U_p\subseteq K$, right multiplication by $g$ defines a map
\[
	H_\ast(X_{U_pU^p},\O) \to H_\ast(X_{(gU_pg^{-1})U^p},\O)
\]
which induces an endomorphism of the projective limit $\Htildestar$, extending the $\O[K]$-action.

\begin{proposition}\cite[Thm.~1.1(1)]{calegari_emerton2012}\label[proposition]{Htildeliggerikategorin}
	With $G$-action defined as above, $\Htildestar \in \Mod_G^{\operatorname{fga}}(\O)$. In particular, $\Htildestar \in \mathfrak{C}_G(\O)$.
\end{proposition}
We now recall the definition of the Hecke algebra which acts on completed homology. For any open normal subgroup $U'_p \subset U_p$, we have a morphism
\[
	\T^S \to \End_{D((\O/\varpi^s)[U_p/U_p'])}(\Cad(U_p'U^p,\O/\varpi^s)),
\]
where $\T^S$ is the abstract Hecke algebra of \Cref{sec_hecke} and $D(-)$ denotes the unbounded derived category. We define $\T^S(U_p'U^p,\O/\varpi^s)$ as the image of this homomorphism. Varying $s$ and $U_p'$, we obtain a projective system, and the \textit{big Hecke algebra} of tame level $U^p$ is defined as the projective limit
\[
	\T^S(U^p) \coloneqq \varprojlim_{s,U_p'} \T^S(U_p'U^p,\O/\varpi^s).
\]
Then $\T^S(U^p)$ acts $G$-equivariantly on completed homology; in other words, there exists a map
\[
	\T^S(U^p) \to \End_{\CGO}\big(\Htildestar\big).
\]
\begin{proposition}\label[proposition]{hyperhomology}
	Let $\sigma=\swtp$. There is a Hecke-equivariant homological spectral sequence
	\[
		E^2_{i,j} = \Tor_j^{\OK}\big(\Htildei, \sigma\big)\implies H_{i+j}(X_{K U^p},\sigma)
	\]
	and an isomorphism
	\[
		H_\ast(X_{K U^p},\sigma) \cong H_\ast(X_{K_\tau U^p},\sigma(\lambda))[\sigma(\tau)^\ast],
	\]
	where $K_\tau=\ker(\sigma(\tau))\subset K$ and $[\sigma(\tau)^\ast]$ denotes the $\sigma(\tau)^\ast$-isotypic vectors.
\end{proposition}
\begin{proof}
The quotient $K/K_\tau$ is a finite group. By \Cref{passagefinitelevel} and our choice of chain homotopy equivalence between the adèlic complex and the Borel--Serre complex, we have isomorphisms
\begin{equation*}
	\begin{split}
		\Htildestar &\cong H_\ast \big(\Cad(K,\Z[K])\otimes_{\Z[K]} \O[[K]]\big) \\
					&\cong H_\ast \big(\CBS(K,\Z[K])\otimes_{\Z[K]} \O[[K]]\big).
	\end{split}
\end{equation*}

Let $\Ctilde=\CBS(K,\Z[K])\otimes_{\Z[K]} \O[[K]]$. This is a complex of finitely generated and projective $\O[[K]]$-modules, so there is a hyperhomology spectral sequence (see \cite[Thm.~5.7.6]{weibel_2011})
\[
	E^2_{i,j} = \Tor_j^{\OK}\big(\Htildei,\sigma\big) \implies H_{i+j}\big(\Ctilde \otimes_{\OK} \sigma(\lambda)\otimes_E \sigma(\mathbf{\tau})\big).
\]
Note that
\[
	\Ctilde \otimes_{\OK} \sigma(\lambda)\otimes_E \sigma(\mathbf{\tau}) \cong \big(\Ctilde\otimes_{\O[[K_\tau]]} \sigma(\lambda)\otimes_E \sigma(\mathbf{\tau})\big)_{K/K_\tau}
\]
where $(-)_{K/K_\tau}$ denotes taking coinvariants with respect to the action defined by
\[
	k(u\otimes u') = uk \otimes k^{-1}u',
\]
where $u\in \widetilde{C}_n$ and $u' \in \sigma(\lambda)\otimes_E \sigma(\mathbf{\tau})$.
In characteristic 0, taking coinvariants is an exact functor, and since $K_\tau$ acts trivially on $\sigma(\tau)$ we have
\[
	H_{i+j}\big(\Ctilde \otimes_{\OK} \sigma(\lambda)\otimes_E \sigma(\mathbf{\tau})\big) \cong H_{i+j}\big(\Ctilde \otimes_{\O[[K_\tau]]} \sigma(\lambda) \otimes_E \sigma(\tau) \big)_{K/K\tau},
\]
which by Schur's lemma is precisely $H_{i+j}(X_{K_\tau U^p},\sigma(\lambda))[\sigma(\tau)^\ast]$. This completes the proof.
\end{proof}
One expects the spectral sequence above to degenerate at the $E^2$-page after localisation at a non-Eisenstein ideal $\m\subset \T^S(U^p)$, meaning that the representation $\rhobar_\m$ introduced in the following section is absolutely irreducible.
\begin{conjecture}\label[conjecture]{Htildevanishq0}
	Let $\m\subset \T^S(U^p)$ be a non-Eisenstein maximal ideal. Then
	\[
		\Htildestar_\m = \Htildeq_\m.
	\]
\end{conjecture}

\begin{remark}
As is shown in \cite[Prop.~4.2.1(1)]{gee_newton_2022}, this vanishing follows from conjectures of Calegari--Emerton \cite[Conj.~1.5]{calegari_emerton2012} and Calegari--Geraghty \cite[Conj.~B(4)(a)]{calegari_geraghty_2018}. These conjectures are open in general, but known to hold when $l_0=1$, i.e. when $F$ is an imaginary quadratic field.
\end{remark}
Thus, if $\m$ is non-Eisenstein, \Cref{Htildevanishq0} and \Cref{hyperhomology} imply an isomorphism
\[
	\Tor_i^{\OK}\big(\Htildeq_\m,\swtp\big) \cong H_{q_0+i}\big(X_{K U^p},\sigma(\lambda,\tau)\big).
\]

\subsection{Galois representations}\label[section]{section_galoisrepresentations}
Suppose that $\Pi$ is a regular algebraic cuspidal automorphic representation of $\G(\A^\infty_F)$ and let $S$ contain the set of ramified places of $\Pi$ and the places above $p$. To the weight $\lambda$ of $\Pi$ and its inertial types $\operatorname{LL}(\Pi_v)|_{\IFv}$ at $v\mid p$ we associate a representation $\sigma \coloneqq \sigma(\lambda,\tau)$ as in \Cref{ssec_homology}, so that $H_i(X_{KU^p},\sigma)_\Pi$ vanishes when $i\notin [q_0,q_0+l_0]$. The Hecke-equivariance of the hyperhomology spectral sequence \Cref{hyperhomology} implies that the action of the abstract Hecke algebra $\T^S$ on the homology $H_\ast(X_{KU^p},\sigma)$ factors through the big Hecke algebra $\T^S(U^p)$, and if we let 
\[
	\p = \p_{\Pi,\iota} = \ker\big(\T^S(U^p) \to \End(H_\ast(X_{KU^p},\sigma)_{\Pi})\big)
\]
then
\[
	H_\ast(X_{U_pU^p},\sigma)_\Pi \cong H_\ast(X_{U_pU^p},\sigma)_{\p}.
\]
Using the main result of \cite{scholze_2015} and \cite[Thm.~6.1.4]{caraiani_etal_2020}, we associate to 
$
	\m = \p + (\varpi) \subset \T^S(U^p)
$
a representation
\[
	\rho_\m \colon \GFS \to \GL_2(\T^S(U^p)_\m)
\]
with determinant $\eps$ and reduction 
\[
	\rhobar_\m \colon \GFS \to \GL_2(\T^S(U^p)/\m) \cong \GL_2(k).
\]
Thus, if $
	\mathcal{S}=(\rhobar_\m, S, \eps,\{D_{\rhobarv}^{\eps,\square}\}_{v\in S}),
$
then $\rho_\m$ and $\p$ define surjections
\[
	\RS \surj \T^S(U^p)_\m \surj \T^S(U^p)_\p/\p \cong E,
\]
corresponding to a representation
$
	\rho = \rho_{\Pi,\iota} \colon \GFS \to \GL_2(E).
$
For every $v\mid p$, the pseudorepresentation $\tr \rhobar_{\m}|_{\GFv}$ defines a block $\B_{\m,v}$ of $\Mod^{\operatorname{pfa}}_{\GL_2(F_v),1}(\O)$, and we define a block of $\CGO$ by forming the tensor product as in \Cref{EBresultB}, i.e. $\B \coloneqq \otimes_{v\mid p} \B_{\m,v}.$
\section{Main results}\label[section]{SistaKapitlet}
In this section, we prove our main results. We keep the notation from the preceding section, and make the following assumptions from now on.
\begin{itemize}
	\item[(i)] $\rhobar_\m\colon\GFS\to \GL_2(k)$ is absolutely irreducible and $\rhobar_\m|_{\Gamma_{F(\zeta_p)}}$ has adequate image \cite[Def.~2.3]{thorne2012}.
	\item[(ii)] For every $v\in S_p$, the restriction $\rho_v$ is irreducible and of $v$-adic Hodge type $(\lambda_v,\tau_v,\eps)$.
	\item[(iii)] For every $v\in S$, the Weil--Deligne representation $\operatorname{WD}(\rho_v)$ is generic.
	\item[(iv)] $\Htildestar_\m=\Htildeq_\m$ (\Cref{Htildevanishq0}). 
\end{itemize}

\subsection{Patching}\label[section]{section.patching}
The goal of this subsection is to prove the following theorem.
\begin{theorem}\label[theorem]{arithmeticaction}
	Assuming (i)-(iv) above and  \Cref{Htildevanishq0,localglobal}, using the same notation as in \Cref{section_galoisrepresentations}, let
	\[
		M_0 = \Hom_{\CBO}\big(P_\B,\Htildeq_\m\big) \in \RModcptEB.
	\]
	Then $\mzerophat$ is isomorphic as an $\EBphat$-module to $(\RSrho)^{\oplus |\B|}$,	where the action of \[\EBphat \cong M_{|\B|}(\Rrhop)\] (see \Cref{EBresultB}(iv)) is given by matrix multiplication.
\end{theorem}
This result is a characteristic 0 analogue of \cite[Conj.~5.1.2]{gee_newton_2022}. Note  that our setting differs in that we have no `minimal level' assumption ensuring the smoothness of the local framed deformation rings $\Rrhobarvbox$ for $v\in S\setminus S_p$. For our purposes, it is sufficient to assume that the associated Weil--Deligne representations of $\rho_v$ are generic at all places $v\in S\setminus S_p$, as it ensures the restrictions $\rho_v$ define smooth points in the generic fibres of $\Rrhobarvbox$ (\Cref{lokal_generisk_glatt}).

Our proof of \Cref{arithmeticaction} uses the Taylor--Wiles method as adapted to completed homology in \cite{gee_newton_2022}. The strategy is to first prove an analogous result at `patched' level (\Cref{patchedarithmeticaction}) and then `unpatch' to deduce \Cref{arithmeticaction}. Before we can state \Cref{patchedarithmeticaction}, we need to recall the construction of patched completed homology. We mostly follow \cite{gee_newton_2022} but have made slight adjustments.

To begin, we note that the assumption that $\rhobar(\Gamma_{F(\zeta_p)})$ is adequate is equivalent to it being enormous (\cite[Lem.~3.2.3]{gee_newton_2022}), and we let $q$ be an integer large enough to guarantee the existence of Taylor--Wiles primes as in \cite[Lem.~3.3.1]{gee_newton_2022}. We let $\mathcal{T}$ be the power series ring over $\O$ in the $S$-frame variables, i.e.
\[
	\mathcal{T} = \O[[\{X^v_{i,j}\mid v \in S,\ i,j=1,2\}]]/(X^{v_0}_{1,1}),
\]
where $v_0$ is an arbitrary element of $S$. Then $\mathcal{T}$ is of relative dimension $4|S|-1$ over $\O$, and we define
\[
	\O_\infty \coloneqq \mathcal{T} \otimeshat_\O \O[[(\Z_p^\times)^q]] \cong \O[[X_1,\dots,X_{4|S|-1+q}]].
\]
Let $\fraka \coloneqq \ker(\O_\infty \to \O)$ be the augmentation ideal of $\O_\infty$, and set
\[
	S_\infty \coloneqq \Rpsp \otimeshat_\O \O_\infty.
\]
Letting $R_\infty \coloneqq \RSlocS[[X_1,\dots,X_{|S|-1+q-[F\colon\Q]-l_0}]]$, the patching argument produces a morphism $\O_\infty \to R_\infty$ in $\CNLk$. Since $R_\infty$ is a $\Rpsp$-algebra, we obtain morphisms
\[
	\Rpsp \otimeshat_\O \O_\infty = S_\infty \to R_\infty \surj \RS \surj \T^S(U^p)_\m \surj E,
\]
where $\mathcal{S}=(\rhobar_\m,S,\eps,\{D^{\eps,\square}_{\rhobarv}\}_{v\in S})$. One expects the penultimate map to be an isomorphism, and we will confirm this expectation after localisation and completion at the point corresponding to $\rho$. Let
\[
	\p_\infty \coloneqq \ker(R_\infty \to E),\quad \q_\infty \coloneqq \ker(S_\infty \to R_\infty \to E),
\]
so that $\q_\infty = \phi^{-1}(\p_\infty)$ and we have a homomorphism $\Sinfphat \to \Rinfphat$.
\begin{lemma}\label[lemma]{SRregular}
	The rings $\Sinfphat$ and $\Rinfphat$ are regular and
	\[
		\dim \Rinfphat = \dim \Sinfphat - [F\colon\Q] - l_0.
	\]	
\end{lemma}
\vspace{-0.5cm}
\begin{proof}
	The rings $R_\infty$ and $S_\infty$ are rings of formal power series with coefficients in $\RSlocS$ and $\Rpsp$, respectively; the regularity of the points $\p_\infty,\q_\infty$ follows from that of the corresponding points of $\RSlocS$ and $\Rpsp$ implied by \Cref{lokal_generisk_glatt,pstringsregular,completion_pseudo} together with genericity of $\operatorname{WD}(\rho_v)$ for $v\in S$ and irreducibility of $\rho_v$ for $v\mid p$. From \cite[Lem.~3.3]{blght} and \Cref{dimensionlokaladefringar}, we get:
	\begin{equation*}
		\begin{split}
			\dim \Rinfphat &= \sum_{v\in S} \dim (\Rrhobarvbox)_{\p_v} + |S|-1+q-[F\colon\Q]-l_0\\
			&= 3|S\setminus S_p| + 3|S_p| + \sum_{v\mid p} \dim \Rrhov + |S|-1+q-[F\colon\Q]-l_0\\
			&= \sum_{v\mid p} \dim \Rrhov + 4|S| - 1 + q - [F\colon\Q] - l_0,\\
			\dim \Sinfphat &= \sum_{v\mid p} \dim \Rrhov + \dim \O_\infty - 1 = \sum_{v\mid p} \dim \Rrhov + 4|S| - 1 + q.
	\end{split}
\end{equation*}
The result follows.
\end{proof}

So far, we have only mentioned the rings involved in the patching argument. Let us now recall the key features for us of the `patched' complex on which these rings act. Note that the `minimal level' assumption present in \cite{gee_newton_2022} is not required to prove the cited results.
\begin{theorem}\label[theorem]{HCtilderesults}
	There exists a perfect complex $\Ctildeinfty$ of $\OK$-modules with $G$-action such that:	
	\begin{itemize}
		\item[(i)] The action of $\O_\infty$ on $\Ctildeinfty$ factors through the map $\O_\infty\to R_\infty$.
		\item[(ii)] $H_\ast(\Ctildeinfty)\in \CGO$.
		\item[(iii)] There is a $G$-equivariant isomorphism of $\O_\infty[[K]]$-modules
		\[
			H_i(\Ctildeinfty \otimes_{\O_\infty} \O_{\infty}/\fraka)\cong \Htildei_\m.
		\]

		\item[(iv)] If \Cref{Htildevanishq0} holds, then $H_\ast(\Ctildeinfty) = H_{q_0}(\Ctildeinfty)$.
	\end{itemize}
\end{theorem}
\begin{proof}
	The complex is constructed in \cite[§3.4]{gee_newton_2022}. For (i), (iii) and (iv), see \cite[Prop.~3.4.16(2), Rem.~3.4.17, Prop.~3.4.19, Prop.~4.2.1]{gee_newton_2022}. To prove (ii), first note that since $\Ctildeinfty$ is a perfect complex of $\O_\infty[[K]]$-modules,
	\[
		H_\ast(\Ctildeinfty) \cong \varprojlim_{n} H_\ast (\Ctildeinfty \otimes_{\Oinfty} \Oinfty/\fraka^n).
	\]
	The category $\CGO$ has projective limits (\cite[p.14]{paskunas_2013}) and hence it suffices to prove that each term in the inverse limit lies in $\CGO$. We proceed by induction on $n$, noting that the case $n=1$ follows from (iii) and \Cref{Htildeliggerikategorin}. For $n\geq 2$, we have a short exact sequence of chain complexes
	\[
		0 \to \Ctildeinfty \otimes_{\Oinfty} \fraka^{n-1}/\fraka^n \to \Ctildeinfty \otimes_{\Oinfty} \Oinfty/\fraka^n \to \Ctildeinfty \otimes_{\Oinfty} \Oinfty/\fraka^{n-1} \to 0
	\]
	inducing a long exact sequence of homology groups
	\[
		\dots \to H_i(\Ctildeinfty \otimes_{\Oinfty} \fraka^{n-1}/\fraka^n) \to H_i(\Ctildeinfty \otimes_{\Oinfty} \Oinfty/\fraka^n) \to H_\ast(\Ctildeinfty \otimes_{\Oinfty} \Oinfty/\fraka^{n-1}) \to \dots.
	\]
	For the induction step, it is enough to prove that the left-hand term lies in $\CGO$, since $\CGO$ is closed under kernels, cokernels and extensions in $\ModpfaGO$. Indeed, $\CGO$ is abelian and the inclusion $\Mod^{\operatorname{ladm}}_{G}(\O)\hookrightarrow \Mod^{\operatorname{sm}}_{G}(\O)$ preserves injectives (see \cite[Cor.~5.18]{paskunas_2013}), so that for any $V,W \in \Mod^{\operatorname{ladm}}_G(\O),$ \[\Ext^1_{\Mod^{\operatorname{ladm}}_G(\O)}(V,W)\cong \Ext^1_{\Mod^{\operatorname{sm}}_G(\O)}(V,W).\] Now, note that
	\[
		\Ctildeinfty \otimes_{\Oinfty} \fraka^{n-1}/\fraka \cong (\Ctildeinfty\otimes_{\Oinfty} \Oinfty/\fraka)^{\oplus r}
	\]
	for some $r$ depending on $n$. The right-hand side has homology equal to a direct sum of copies of $\Htildeq_\m$, and the theorem now follows from the long exact sequence and induction on $n$.
\end{proof}

Neither $\Htildeq_\m$ nor $\HCtildeinfty$ is finitely generated over $\T^S(U^p)$, so to carry out the depth estimate part of the Taylor--Wiles method, we work instead with their images in $\RModcptEB$; let
\begin{align*}
	M_0 &\coloneqq \Hom_{\CGO}\big(P_\B,\Htildeq_\m\big), \\
	M_\infty &\coloneqq \Hom_{\CGO}\big(P_\B,\HCtildeinfty\big). 
\end{align*}
This idea can only work under the following additional assumption. 
\begin{conjecture}\label[conjecture]{localglobal}
	The two actions of $\Rpsp$ on $M_\infty$ -- one coming from the map $\Rpsp \to \EB$ and the other from the map $\Rpsp\to\T^S(U^p)_\m$ -- coincide.
\end{conjecture}
\begin{remark}\label[remark]{rmk_pan_localglobal}
	In \cite[§3.5]{pan_2022}, a similar statement is proved in a different setting using a local-global compatibility result. While we have not checked the details, it seems a reasonable guess that the conjecture can be verified in our setting using Hevesi's local-global compatibility result \cite{hevesi2024} and the argument in \cite[Prop.~5.3.1]{gee_newton_2022}, under some simplifying assumptions on $\rhobar_\m|_{\GFv}$ for $v\in S_p$ \cite[p.30]{gee_newton_2022}.
\end{remark}
\textbf{From now on, we assume \Cref{localglobal}}. 
Note that, as an $\Rpsp$-module,
\[
	M_0 \cong \bigoplus_{\pi\in \B} \Hom_{\CGO}\big(P_\pi,\Htildeq_\m\big)
\]
and similarly for $M_\infty$. By \Cref{HCtilderesults}(i) and \Cref{localglobal}, $M_\infty$ is an $S_\infty$-module such that the action of $S_\infty$ factors through $S_\infty \to R_\infty$. 

\begin{proposition}\label[proposition]{m0minftyprop}
	With $M_0,M_\infty$ as above, we have:
	\begin{itemize}
		\item[(i)] $M_\infty \otimesL_{\O_\infty} \O_\infty/\fraka \cong M_0$. 
		\item[(ii)] $M_0$ is a finitely generated $\Rpsp$-module.
		\item[(iii)] $M_\infty$ is finitely generated as an $S_\infty$-module and as an $R_\infty$-module.
	\end{itemize}
\end{proposition}
\begin{proof}
	(i) follows from \Cref{HCtilderesults}. Note that since $\Htildeq_\m$ is finitely generated over $\O[[K]]$, $M_0$ is finitely generated over $\EB$ by \Cref{OKändligEBändlig}. Now, $\EB$ is finitely generated over $\Rpsp$ (\Cref{EBresultB}) and (ii) follows. Since the $S_\infty$-action factors through $S_\infty\to R_\infty$, the first statement implies the second. From (i) we know that $M_\infty\otimes_{\Oinfty} \O \cong M_0$, and therefore by Nakayama's lemma for compact modules \cite[Cor.~1.5]{brumer1966} it suffices to note that  $M_\infty$ is compact module. Thus, we have proved (iii).
\end{proof}

We obtain a corresponding diagram to what we had before:
\[
	S_\infty \to R_\infty \to \End_{\EB}(M_\infty).
\]
Localising and completing the diagram we obtain
\[
	\Sinfphat \to \Rinfphat \to \End_{\EBphat}(\minfphat).
\]
\begin{lemma}
	The module $\minfphat$ is finitely generated over $\Rinfphat$ and $\Sinfphat$.
\end{lemma}
\begin{proof}
	The first statement follows from the fact that $M_\infty$ is finitely generated over $R_\infty$, which is \Cref{m0minftyprop}(ii). For the second, note that since $\p_\infty$ is maximal in $R_\infty[1/\varpi]$, we have
	\[
		\minfphat \cong \varprojlim_r M_\infty[1/\varpi]/\p_\infty^r.
	\]
	Now, by definition $\q_\infty = \phi^{-1}(\p_\infty)$ and hence $\phi(\q_\infty^r)\subseteq \p_\infty^r$. It follows that the finitely generated $\Sinfphat$-module $(M_\infty)_{\q_\infty}^\wedge$ surjects onto $\minfphat$, and the statement follows.
\end{proof}
The following theorem is the patched counterpart of \Cref{arithmeticaction}.
\begin{theorem}\label[theorem]{patchedarithmeticaction}
	$\minfphat$ is a finitely generated and free $\Rinfphat$-module.
\end{theorem}

\begin{proof}
	Let $\sigma \coloneqq \bigotimes_{v\mid p}\ \sigma(\lambda_v,\tau_v)$ be the $K$-module for which $\Pi$ contributes (in degrees $[q_0,q_0+l_0]$) to $H_\ast(X_{KU^p},\sigma)$. Define
	\[
		\Rpsigma = \big(\widehat{\bigotimes}_{v\mid p} \Rvsv \big)_\p^\wedge
	\]
where $\p$ is the ideal corresponding to $\rho$, and set
\[
	\asigmap = \ker\big(\Sinfphat \surj \Rpsigma \big).
\]
By \Cref{pstringsregular} and \Cref{SRregular}, the rings $\Sinfphat$ and $\Rpsigma$ are regular. Thus, by \Cref{SRregular}, $\asigmap$ is generated by a regular sequence of length
\[
	\dim \Sinfphat - \dim \Rpsigma = \dim \Sinfphat - [F\colon\Q]. 
\]
We are going to combine \Cref{char0q0l0} and \Cref{depth-l} to prove that $\minfphat$ is maximal Cohen--Macaulay over the regular local ring $\Rinfphat$.
\begin{lemma}\label[lemma]{matrislemmat}
	We have a canonical isomorphism
	\[
		\big(\mzerophat \otimesL_{\EBphat} \PB(\sigma)_\p^\wedge)\big)^{\oplus |\B|^2} \cong \mzerophat \otimesL_{\Rrhop} \PB(\sigma)_\p^\wedge.
	\]
\end{lemma}
\begin{proof}
	Since $P_\B = \bigoplus_{\pi\in\B} P_\pi$, we have an isomorphism of $\Rpsp$-modules
	\[
		\mzero \cong \bigoplus_{\pi\in\B} \Hom_{\CGO}(P_\pi,\Htildeq_\m).
	\]
	By \Cref{EBresultB}, $\EBphat \cong M_{|\B|}(\Rrhop)$ and consequently $\mzerophat$ is a direct sum of $|\B|$ pairwise isomorphic $\Rrhop$-modules. The same is true for $\PB(\sigma)_\p^\wedge$, and hence
	\begin{equation*}
		\begin{split}
			\mzerophat \otimesL_{\Rrhop} \PB(\sigma)_\p^\wedge 
			&\simeq 
			\bigoplus_{\pi_1\in\B}\bigoplus_{\pi_2\in\B} \Hom_{\CGO}\big(P_{\pi_1},\Htildeq_\m\big)_\p^\wedge \otimesL_{\Rrhop} P_{\pi_2}(\sigma)_\p^\wedge\\ 
			&\simeq 
			\bigoplus_{\pi_1\in\B}\bigoplus_{\pi_2\in\B} \mzerophat \otimesL_{\EBphat} \PB(\sigma)_\p^\wedge.
		\end{split}
	\end{equation*}
	Now, for any commutative ring $R$, matrix algebra $E=M_n(R)$, and $R$-modules $M,N$, there is an equivalence
	$
		M \otimesL_R N \simeq M^{\oplus n} \otimesL_{E} N^{\oplus n},
	$
	where $M^{\oplus n},N^{\oplus n}$ are acted on by left- and right matrix multiplication, respectively. It follows that
	\[
		\bigoplus_{\pi_1\in\B}\bigoplus_{\pi_2\in\B} \mzerophat \otimesL_{\EBphat} \PB(\sigma)_\p^\wedge \simeq \big(\mzerophat \otimesL_{\EBphat} \PB(\sigma)_\p^\wedge\big)^{\oplus |\B|^2},
	\]
	as claimed.
\end{proof}
We now return to the proof of \Cref{patchedarithmeticaction}. \Cref{m0minftyprop} implies the quasi-isomorphism
\[
	\minfphat \otimesL_{\Sinfphat} \Sinfphat / \asigmap \simeq \mzerophat \otimesL_{\Rrhop} \Rpsigma,
\]
which by \Cref{Ppcoinvariants} and \Cref{matrislemmat} is equivalent to
\[
	\big((\mzero \otimesL_{\EB} P_\B \otimesLOK \sigmao )_\p^\wedge\big)^{\oplus |\B|^2} \cong \big((\Htildeq_\m \otimesLOK \sigmao)_\p^\wedge\big)^{\oplus |\B|^2}.
\]
By \Cref{hyperhomology} and \Cref{char0q0l0}, the homology of this complex is concentrated in degrees $[q_0,q_0+l_0]$, and thus we have proved
\[
	\Tor_i^{\Sinfphat}\big(\minfphat,\Rpsigma\big) = 0 \text{ for } i\notin [0,l_0].
\]
Applying \Cref{depth-l}, we obtain
\[
	\dpt_{\Sinfphat} \minfphat \geq \dpt_{\asigmap} \minfphat
	= \dim \Sinfphat - [F\colon\Q] - l_0 
	= \dim \Rinfphat.
\]
Since the action of $\Sinfphat$ on $\minfphat$ factors through $\Rinfphat$, any regular sequence in $\Sinfphat$ gives rise to a regular sequence in $\Rinfphat$, and hence
\[
	\dpt_{\Rinfphat} \minfphat \geq \dpt_{\Sinfphat} \minfphat \geq \dim \Rinfphat.
\]
Thus, $\minfphat$ is a maximal Cohen--Macaulay module over the regular local ring $\Rinfphat$, and is therefore free by \Cref{maxCMmeansfree}. Thus, we have proved \Cref{patchedarithmeticaction}.
\end{proof}
Having proved \Cref{patchedarithmeticaction}, we now deduce \Cref{arithmeticaction}.
\begin{corollary}\label[corollary]{RlikaT}
	The natural surjections
	\[
		(R_\infty/\fraka)_{\p_\infty}^\wedge \surj \RSrho \surj \T^S(U^p)_\p^\wedge
	\]
	are isomorphisms. In particular, \Cref{arithmeticaction} holds.
\end{corollary}
\begin{proof}
	By \Cref{patchedarithmeticaction}, $\minfphat$ is free over $\Rinfphat$, and hence the unpatched module
	\[
		\mzerophat = \minfphat \otimes_{\O_\infty} \O_\infty/\fraka
	\]
	is free over $\Rinfphat/\fraka$. But the action of $\Rinfphat/\fraka$ on $\mzerophat$ factors through the maps displayed in the statement of the theorem, which forces them to be injective. Thus both maps are bijections and $\mzerophat$ is free over $\Rinfphat/\fraka \cong \RSrho$.
\end{proof}

\subsection{The main theorem}
In this final section, we deduce our main result from \Cref{patchedarithmeticaction}, still keeping the notation of \Cref{section_galoisrepresentations}.
\begin{theorem}\label[theorem]{maintheorem}
	 Let $F$ be a CM field, suppose $p\geq5$ is totally split in $F$ and let $\Pi$ be a regular algebraic cuspidal automorphic representation of $\G=\PGL_2/F$ of weight $\lambda = (\lambda_{\iota,v})\in (\Z^2)^{\Hom(F,\C)}$, and fix an isomorphism $\iota\colon\overline{\Q}_p\to \C$. Let $\tau=(\tau_v)_{v\mid p} = (\operatorname{LL}(\Pi_v)|_{\IFv})_{v\mid p}$ be the inertial types of $\Pi$ and suppose $KU^p \subset \G(\A^\infty_F)$ is a good subgroup. Let $\m=\m_{\Pi,\iota}\subset \T^S(U^p)$ be the maximal ideal of the big Hecke algebra associated to $(\Pi,\iota)$, $\rho_\m\colon\GFS\to \GL_2(\T^S(U^p)_\m)$ the representation associated to $\m$, and $\rho=\rho_{\Pi,\iota}\colon\GFS\to \GL_2(E)$ the characteristic 0 representation associated to $(\Pi,\iota)$.
	 
	Suppose the following statements hold.

	\begin{itemize}
		\item[(i)] $\rhobar_\m\colon \GFS\to \GL_2(k)$ is absolutely irreducible and the restriction $\rhobar_\m|_{\Gamma_{F(\zeta_p)}}$ has adequate image.
		\item[(ii)] For every $v\in S_p$, $\rho_v$ is irreducible and, for any $\iota\in \Hom(F,\C)$, of $v$-adic Hodge type $(\lambda_{\iota,v},\tau_v,\eps)$.
		\item[(iii)] For every $v\in S$, the Weil--Deligne representation $\operatorname{WD}(\rho_v)$ is generic.
		\item[(iv)] Local-global compatibility at $S_p$ in the sense of \Cref{localglobal} holds.
		\item[(v)] $\Htildestar_\m$ vanishes outside degree $q_0$ (\Cref{Htildevanishq0}).
		\item[(vi)] The adjoint Bloch--Kato Selmer group vanishes, i.e. $H^1_g(\GFS,\ad^0\rho)=0$. 
	\end{itemize}
	We define a deformation problem $\mathcal{S}=(\rhobar_\m, S, \eps,\{D_{\rhobarv}^{\eps,\square}\}_{v\in S})$ with associated deformation ring $R_{\mathcal{S}}$. Then $\rho$ defines a point $\p$ of $\RS$, and we let $\RSrho = (R_{\mathcal{S}})_{\p}^\wedge$, furthermore $\Rrhop = \hatotimes_{v\mid p} R_{\rho_v}$ and $\Rpsigma = \hatotimes_{v\mid p} R_{\rho_v}(\sigma_v)$. Then there is a natural free action
	\[
		\Tor_\ast^{\Rrhop}\big(\RSrho,\Rpsigma\big) \circlearrowright H_{\ast}(X_{KU^p},\sigma)_\Pi.
	\]
	If, in addition,
	\begin{itemize}
		\item[(vii)] $H^2(\GFS,\ad^0\rho)=0$, i.e. $\RSrho$ is formally smooth,
	\end{itemize}
	then there is a canonical isomorphism of graded-commutative rings
	\[
		\Tor_\ast^{\Rrhop}\big(\RSrho,\Rpsigma\big) \cong \bigwedge^\ast \Hfd.
	\]
\end{theorem}
\begin{remark}
	Let us describe the various roles of the assumptions in \Cref{maintheorem}, and what is known.
\begin{itemize}
	\item \emph{Taylor--Wiles method:} Assumption (i) is necessary in the construction of patched completed homology in \cite{gee_newton_2022}. An assumption of this type features in all known Taylor--Wiles arguments, though it is conceivable that it can be weakened.
	\item \emph{Local-global compatibility:} Assumptions (ii), (iii) and (iv) are forms of local-global compatibility which are expected to hold in general; (ii) is known in many cases by the main result of \cite{hevesi2024}, and ensures that $\rho$ has the `expected' Hodge--Tate weights, so that the action of $\Rrhop$ factors through the quotient $\Rpsigma$. Assumption (iii) was discussed in \Cref{rmk_pan_localglobal}. Assumption (iv) would follow from a strong form of local-global compatibility at the ramified places; in the conjugate self-dual case it follows from the main results of \cite{caraiani2014,caraiani2012}.
	\item \emph{Selmer group vanishing:} Assumption (vi) is known in many cases by the main theorem of \cite{acampo2024}.
	\item \emph{Jannsen's conjecture:} Assumption (vii) is predicted to always hold by a conjecture of Jannsen \cite{jannsen_2010}. We use it to simplify the computation of $\Tor_\ast^{\Rrhop}\big(\RSrho,\Rpsigma\big)$.
\end{itemize}
\end{remark}
\begin{proof}
	By assumption (v), the spectral sequence of \Cref{hyperhomology} degenerates at the $E^2$-page, whence
	\[
		H_{\ast}(X_{KU^p},\sigma)_\Pi \cong H_\ast\bigg(\big(\Htildeq_\m \otimesLOK \sigma\big) \otimesL_{\Rpsp} (\Rpsp)_\p^\wedge\bigg). 
	\]
	Now, $\Htildeq_\m \cong M_0 \otimesL_{\EB} \PB$ where $M_0 = \Hom_{\CGO}(\PB,\Htildeq_\m)$, and hence
\begin{equation*}
	\begin{split}
		\big(\Htildeq_\m \otimesLOK \sigma\big) \otimesL_{\Rpsp} (\Rpsp)_\p^\wedge 
		&\simeq \big(M_0 \otimesL_{\EB} \PB \otimesLOK \sigma\big)\otimesL_{\Rpsp} (\Rpsp)_\p^\wedge \\
		&\simeq \mzerophat \otimesL_{\EBphat} \PB(\sigma)_\p^\wedge.
	\end{split}
\end{equation*}
By \Cref{PBpcoinvariants,arithmeticaction,matrislemmat}, we have an equivalence
\[
	\mzerophat \otimesL_{\EBphat} \PB(\sigma)_\p^\wedge \simeq \RSrho \otimesL_{\Rrhop} \Rpsigma,
\]
and the first part of the theorem now follows, i.e.
\begin{equation*}
	\begin{split}
		H_{q_0+i}(X_{KU^p},\sigma)_\p 	&\cong \Tor_i^{\Rrhop}\big(\RSrho, \Rpsigma\big),
	\end{split}
\end{equation*}
This proves the first part of the theorem. Under the additional assumption that $H^2(\GFS,\ad^0\rho)=0$, all three of $\Rrhop,\RSrho$ and $\Rpsigma$ are formally smooth $E$-algebras, and
\[
	\Tor_\ast^{\Rrhop}\big(\RSrho,\Rpsigma\big) \cong \bigwedge^\ast \Tor_1^{\Rrhop}\big(\RSrho,\Rpsigma\big).
\]
Moreover, the exact sequence of \Cref{tangentspacesequence} simplifies to a short exact sequence
\[
	0\to \H \to \prod_{v\mid p}\frac{\Hp}{\Hfp} \to \Hfd^\vee \to 0,
\]
and the final term has dimension $h^1_f(\GFS,\ad^0\rho(1))=l_0$ by \Cref{complementarydimension}. Hence, \[\m_{\Rrhop} = \ker(\Rrhop\to \RSrho)+\ker(\Rrhop\to\Rpsigma),\] and taking duals, we obtain a natural identification between $\Hfd$ and the intersection
\[
	\ker(\Rrhop\to \RSrho)/\m_{\Rrhop}^2 \cap \ker(\Rrhop\to\Rpsigma)/\m_{\Rrhop}^2 \subset \m_{\Rrhop}/\m_{\Rrhop}^2,
\]  
from which we obtain a canonical isomorphism
\[
	\Hfd \cong \Tor_1^{\Rrhop}\big(\RSrho,\Rpsigma\big),
\]
by \cite[Lem.~2.2]{hansen_thorne2017} (and using \Cref{TorAmodI} to eliminate variables).
\end{proof}

\bibliographystyle{alpha}
\bibliography{bib}

\end{document}